\documentclass[11pt]{article}
\usepackage{amsmath,amssymb,mathdots,array,booktabs,graphicx,mathrsfs,amsthm}
\usepackage[T1]{fontenc}
\usepackage{color}
\usepackage{kbordermatrix}

\newtheorem{theorem}{Theorem}[section]
\newtheorem{lemma}[theorem]{Lemma}
\newtheorem{corollary}[theorem]{Corollary}
\newtheorem{proposition}[theorem]{Proposition}

\theoremstyle{definition}
\newtheorem{definition}[theorem]{Definition}
\newtheorem{example}{Example}[section]

\theoremstyle{remark}
\newtheorem{remark}[theorem]{Remark}

\numberwithin{equation}{section}

\newcommand{\ww}{\widetilde{W}}
\newcommand{\na}{\mathscr{N}(\mathcal{A})}
\newcommand{\bsmat}{\left[\begin{smallmatrix} }
\newcommand{\esmat}{\end{smallmatrix}\right] }
\newcommand{\E}{\mathbb{E}}

\DeclareMathOperator*{\argmin}{argmin}

\DeclareMathOperator{\berr}{bker}
\DeclareMathOperator{\diag}{diag}
\DeclareMathOperator{\card}{card}
\DeclareMathOperator{\cond}{cond}

\DeclareMathOperator{\reshape}{reshape}
\DeclareMathOperator{\sep}{sep}
\DeclareMathOperator{\subspan}{span}

\DeclareMathOperator{\trace}{trace}
\DeclareMathOperator{\myvec}{vec}

\DeclareMathOperator{\Bdiag}{Bdiag}
\DeclareMathOperator{\OffBdiag}{OffBdiag}

\DeclareMathOperator{\robu}{nd}
\DeclareMathOperator{\uniq}{uq}
\DeclareMathOperator{\ub}{ub}

\DeclareMathOperator{\p}{p}

\DeclareMathOperator{\F}{F}
\DeclareMathOperator{\HH}{H}
\DeclareMathOperator{\T}{T}

\def\cA{\mathcal{A}}
\def\cE{\mathcal{E}}

\def\gjbdp{{\sc gjbdp}}
\def\jbd{{\sc jbd}}
\def\jbdp{{\sc jbdp}}
\def\ojbdp{{\sc o-jbdp}}

\def\jdp{{\sc jdp}}
\def\ojdp{{\sc o-jdp}}

\def\wtd{\widetilde}
\def\what{\widehat}

\def\sss{\scriptscriptstyle}
\def\hm{\hphantom{-}}

\title{Perturbation Analysis for Matrix  Joint Block Diagonalization}

\author{
Yunfeng Cai%
\thanks{LMAM \& School of Mathematical Sciences, Peking Univ., Beijing, P.R. China, 100871. 
        Email: {\tt yfcai@math.pku.edu.cn}.
             The work of this author  was supported in part by NSFC grants 11301013, 11671023 and 11421101. 
             }
\and
Ren-Cang Li%
\thanks{
           Department of Mathematics,
           University of Texas at Arlington,
           P.O. Box 19408,
           Arlington, TX 76019-0408, USA. Email: {\tt  rcli@uta.edu}.
      The work of this author was supported in part by
      NSF grants  DMS-1317330 and CCF-1527104, and NSFC grant 11428104.}
}

\begin{document}

\maketitle

%\date{March 1, 2017}

\begin{abstract}
The matrix joint block diagonalization problem (\jbdp) of a given matrix set
$\cA=\{A_i\}_{i=1}^m$ is about finding a nonsingular matrix $W$ such that all
$W^{\T} A_i W$ are block diagonal. It includes the matrix joint diagonalization problem (\jdp) as a special case for which
all $W^{\T} A_i W$ are required diagonal. Generically, such a matrix $W$ may not exist, but
there are practically applications such as multidimensional independent component analysis (MICA)
for which it does exist under the ideal situation, ie., no noise is presented.
However, in practice noises do get in and, as a consequence,
the matrix set is only approximately block diagonalizable, i.e., one can only make all
$\wtd W^{\T} A_i\wtd W$ nearly block diagonal at best, where $\wtd W$ is an approximation to $W$, obtained
usually by computation. This motivates us to develop a perturbation theory
for \jbdp\ to address, among others, the question: how accurate this $\wtd W$ is.
Previously such a theory for \jdp\
has been discussed, but no effort has been attempted for \jbdp\  yet.
In this paper, with the help of a necessary and sufficient condition
for solution uniqueness of \jbdp\  recently developed in [Cai and Liu, {\em SIAM J. Matrix Anal. Appl.}, 38(1):50--71, 2017],
we are able to establish an error bound, perform backward error analysis, and propose a condition number for
\jbdp.
Numerical tests validate the theoretical results.
\end{abstract}

%\centerline{\em Draft: not for wide distribution}

%\medskip
{\small

}

\section{Introduction}
%\setpagewiselinenumbers
%\linenumbers
The matrix joint block diagonalization problem (\jbdp) is about jointly block diagonalizing a set of matrices.
In recent years, it has found many applications in
independent subspace analysis, also known as {\em multidimensional independent component analysis\/} (MICA) (see,
e.g., \cite{cardoso1998multidimensional, de2000fetal,  theis2005blind, theis2006towards})
and semidefinite programming (see, e.g., \cite{ bai2009exploiting, de2007reduction,  de2010exploiting,gatermann2004symmetry}).
Tremendous efforts have been devoted to solving \jbdp\  and, as a result, several numerical methods have been proposed.
The purpose of this paper, however, is to develop
a perturbation theory for \jbdp. For this reason,
we will not delve into  numerical methods, but
refer the interested reader to \cite{cai2017algebraic, chabriel2014joint, de2009survey,  tichavsky2014non}
and references therein.
The {\sc matlab} toolbox for tensor computation -- {\sc tensorlab} \cite{tensorlab} can also be used for the purpose.

In the rest of this section, we will formally introduce \jbdp\ and formulate its associated perturbation problem,
along with some notations and definitions. Through a case study on the basic MICA model,
we rationalize our formulations and provide our motivations for current study in this paper.
Previously, there are only a handful papers in the literature that studied the perturbation analysis of
the matrix joint diagonalization problem (\jdp).  Briefly, we will review these existing works and their limitations.
Finally, we explain our contribution and the organization of  this paper.

\subsection{Joint Block Diagonalization (\jbd)}\label{ssec:jbd}
A {\em partition} of positive integer $n$:
\begin{equation}\label{eq:tau-n}
\tau_n=(n_1,\dots,n_t)
\end{equation}
means that $n_1,n_2,\dots,n_t$ are all positive integers and their sum is $n$, i.e., $\sum_{i=1}^t n_i=n$.
The integer $t$ is called the {\em cardinality} of the partition $\tau_n$,
denoted by $\card(\tau_n)$.
%The set of all partitions of $n$ is denoted by $\mathbb{T}_n$.
%\end{definition}

%\begin{definition}
Given a partition $\tau_n$ as in \eqref{eq:tau-n} and a matrix $A\in\mathbb{R}^{n\times n}$ (the set of $n\times n$ real matrices), we partition
$A$ by
\begin{equation}\label{eq:tau-n-part}
A=\kbordermatrix{ &\sss n_1 & \sss n_2 & \sss \cdots &\sss n_t \\
         \sss n_1 & A_{11} & A_{12} & \cdots & A_{1t} \\
         \sss n_2 & A_{21} & A_{22} & \cdots & A_{2t} \\
         \sss \vdots & \vdots & \vdots &  & \vdots \\
         \sss n_t & A_{t1} & A_{t2} & \cdots & A_{tt} }
\end{equation}
and define its {\em $\tau_n$-block diagonal part\/} and {\em $\tau_n$-off-block diagonal part\/} as
\begin{equation*}
    \Bdiag_{\tau_n}(A)=\diag(A_{11},\dots,A_{tt}),\quad   \OffBdiag_{\tau_n}(A)=A-\Bdiag_{\tau_n}(A).
\end{equation*}
The matrix $A$ is referred to as {\em a $\tau_n$-block diagonal matrix\/} if $\OffBdiag_{\tau_n}(A)=0$.
 The set of all $\tau_n$-block diagonal matrices is denoted by $\mathbb{D}_{\tau_n}$.
%\end{definition}

\smallskip
\noindent{\bf The Joint Block Diagonalization  Problem} (\jbdp).
Let $\cA=\{A_i\}_{i=1}^m$ be the set of $m$ matrices, where each $A_i\in{\mathbb R}^{n\times n}$. The \jbdp\ for
$\cA$ with respect to $\tau_n$ is to
find a nonsingular matrix $W\in\mathbb{R}^{n\times n}$ such that all $W^{\T} A_i W$
are $\tau_n$-block diagonal, i.e.,
\begin{equation}\label{eq:nojbd}
W^{\T} A_i W=\diag(A_{i}^{(11)},\dots, A_{i}^{(tt)}) \quad \mbox{for}\quad i=1,2,\dots, m,
\end{equation}
where $A_{i}^{(jj)}\in\mathbb{R}^{n_j\times n_j}$.
When \eqref{eq:nojbd} holds, we say that $\cA$ is {\em $\tau_n$-block diagonalizable\/} and
$W$ is a {\em $\tau_n$-block diagonalizer\/} of  $\cA$.
If  $W$ is also required to be orthogonal,
this \jbdp\   is referred to as an {\em orthogonal \jbdp\  \/} ({\sc o-jbdp}).

By convention, if $\tau_n=(1,1,\dots,1)$, the word ``{\em $\tau_n$-block\/}'' is
dropped from all relevant terms. For example, ``$\tau_n$-block diagonal'' is reduced to just ``diagonal''.
Correspondingly, the letter ``{\sc B}''
is dropped from all abbreviations. For example, ``\jbdp'' becomes ``\jdp''. This convention is adopted throughout
this article.

Generically, \jbdp\ often has no solution for $m\ge 3$ and $n_j$ not so unevenly distributed, simply by counting the number of
equations implied by \eqref{eq:nojbd} and the number of unknowns. For example, when $m=3$ and $n_1=n_2=n_3=n/3$,
there are $m(n^2-\sum_{i=1}^{t}n_i^2)=2n^2$ equations but only $n^2$ unknowns in $W$. However, in certain practical applications
such as MICA without noises, solvable \jbdp\ do arise.

%The nonsingular matrix
%The \jbdp\   is referred to as an {\em exact/approximate \jbdp\  \/}
%if \eqref{eq:nojbd} is satisfied exactly/approximately.
%The solution to the exact/approximate \jbdp\   is then referred to as an {\em exact/approximate diagonalizer}.
%By taking $\tau_n=(1,1,\ldots,1)$, the {\sc no-jbdp} becomes the so called non-orthogonal joint diagonalization problem ({\sc no-jd}).

\begin{definition}
A permutation matrix $\Pi\in\mathbb{R}^{n\times n}$ is called {\em $\tau_n$-block diagonal preserving\/} if
$\Pi^{\T} D \Pi\in\mathbb{D}_{\tau_n}$ for any $D\in\mathbb{D}_{\tau_n}$.
The set of all $\tau_n$-block diagonal preserving permutation matrices is denoted by $\mathbb{P}_{\tau_n}$.
\end{definition}

Evidentally, any permutation matrix $\Pi\in{\mathbb D}_{\tau_n}$ is in $\mathbb{P}_{\tau_n}$. This is because
such a $\Pi$ can be expressed as $\Pi=\diag(\Pi_1,\ldots,\Pi_t)$, where $\Pi_j$ is an $n_j\times n_j$ permutation matrix.
But not all $\Pi\in\mathbb{P}_{\tau_n}$ also belong to ${\mathbb D}_{\tau_n}$. For example,
for $n=4$ and $\tau_4=(2,2)$,
$\Pi=\begin{bmatrix}
   0 & I_2 \\
   I_2 & 0
\end{bmatrix}\in\mathbb{P}_{\tau_4}$ but $\Pi\not\in{\mathbb D}_{\tau_4}$.
In particular, any permutation matrix $\Pi\in\mathbb{R}^{n\times n}$ is in $\mathbb{P}_{\tau_n}$ when $\tau=(1,1,\ldots,1)$.
It can be proved that for given $\Pi\in\mathbb{P}_{\tau_n}$, there is a permutation $\pi$ if $\{1,2,\ldots,t\}$ such that
$$
\Pi^{\T} D \Pi\in\mathbb{D}_{\tau_n}=\diag(\Pi_1^{\T}D_{\pi(1)}\Pi_1,\Pi_2^{\T}D_{\pi(2)}\Pi_2,\ldots,\Pi_t^{\T}D_{\pi(t)}\Pi_t)
$$
for any
$D=\diag(D_1,D_2,\ldots,D_t)\in\mathbb{D}_{\tau_n}$. Specifically, the subblocks of $\Pi$, if partitioned as in \eqref{eq:tau-n-part},
are all $0$ blocks, except those at the positions $(\pi(j),j)$, which are $n_j\times n_j$ permutation matrices $\Pi_j$.
As a consequence, $n_j=n_{\pi(j)}$ for all $1\le j\le t$.

It is not hard to verify that if $W$ is a $\tau_n$-block diagonalizer of $\cA$, then so is
$WD\Pi$ for any given $D\in\mathbb{D}_{\tau_n}$ and $\Pi\in\mathbb{P}_{\tau_n}$.
In view of this, $\tau_n$-block diagonalizers, if exist, are not unique because any diagonalizer
brings out a class of {\em equivalent\/} diagonalizers in the form of $WD\Pi$.
For this reason, we introduce the following definition for uniquely block diagonalizable \jbdp.

\begin{definition}\label{dfn:us-jbdp}
Two $\tau_n$-block diagonalizers $W$ and $\wtd W$ of $\cA$ are {\em equivalent\/} if
there exist a nonsingular matrix $D\in\mathbb{D}_{\tau_n}$ and  $\Pi\in\mathbb{P}_{\tau_n}$
such that $\wtd W=WD\Pi$.
The \jbdp\  for $\cA$  is said {\em uniquely $\tau_n$-block diagonalizable\/}
if it has a $\tau_n$-block diagonalizer and if any two of its $\tau_n$-block diagonalizers are equivalent.
\end{definition}

To further reduce freedoms for the sake of comparing two diagonalizers,
we restrict our considerations of block diagonalizers to the matrix set:
\begin{equation}
\mathbb{W}_{\tau_n}:=\{W\in\mathbb{R}^{n\times n}\; : \; \mbox{$W$ is nonsingular and $\Bdiag_{\tau_n}(W^{\T}W)=I_n$}\}.
\end{equation}
This doesn't loss any generality because $W[\Bdiag_{\tau_n}(W^{\T}W)]^{-1/2}\in\mathbb{W}_{\tau_n}$
for any nonsingular $W\in\mathbb{R}^{n\times n}$.

\subsection{Perturbation Problem for \jbdp}
%Given a partition $\tau_n\in\mathbb{T}_n$ and a matrix set $\cA=\{A_i\}_{i=1}^m$.
Let $\wtd\cA=\big\{\wtd A_i\big\}_{i=1}^m=\{A_i+\Delta A_i\}_{i=1}^m$,
where $\Delta A_i$ is a perturbation to $A_i$.
Assume $\cA=\{A_i\}_{i=1}^m$ is $\tau_n$-block diagonalizable and  $W\in\mathbb{W}_{\tau_n}$
is  a $\tau_n$-block diagonalizer and \eqref{eq:nojbd} holds. Let
$\ww\in\mathbb{W}_{\tau_n}$ be an approximate $\tau_n$-block diagonalizer of $\wtd\cA$
in the sense that all $\ww^{\T} \wtd A_i \ww$ are approximately $\tau_n$-block diagonal.
%which may or may not be solvable (i.e., all $\ww^{\T} \wtd A_i \ww$ are all approximately $\tau_n$-block diagonal).
How much does $\wtd W$ differ from the block diagonalizer $W$ of $\cA$?

There are two important aspects that needs clarification regarding this perturbation problem.
First, $\wtd\cA$ may or may not be $\tau_n$-block diagonalizable.
Although allowing this counters the common sense that one can only gauge the difference between
diagonalizers that exist, it is for a good reason and important practically to allow this.
As we argued above, a generic \jbdp\ is usually
not block diagonalizable, and thus even if the \jbdp\ for $\cA$ has a diagonalizer, its arbitrarily perturbed problem is
potentially not block diagonalizable no
matter how tiny the perturbation may be. This  leads to an impossible task: to compare
the block diagonalizer $W$ of the unperturbed $\cA$, that does exist, to a diagonalizer $\ww$ of the perturbed matrix set $\wtd\cA$,
that may not exist.
We get around this dilemma by talking about an approximate diagonalizer of $\wtd\cA$, that always exist.
It turns out this workaround is exactly what some practical applications calls for because
most practical \jbdp\ come from block diagonalizable \jbdp\ but contaminated with noises to become
approximately block diagonalizable
and an approximate diagonalizer for the noisy \jbdp\  gets
computed numerically. In such a scenario, it is important to get a sense as how far the computed diagonalizer is from
the exact diagonalizer of the clean albeit unknown \jbdp, had the noises not presented.

The second aspect is about what metric to use in order to measure the difference between two block diagonalizers, given
that they are not unique. In view of Definition~\ref{dfn:us-jbdp} and
the discussion in the paragraph immediately proceeding it, we propose to use
\begin{equation}\label{eq:measure}
\min_{D\in\mathbb{D}_{\tau_n},  \Pi\in\mathbb{P}_{\tau_n}}\frac{\|W - \ww D\Pi\|}{\|\ww\|}
\end{equation}
for the purpose, where $\|\cdot\|$ is some matrix norm. Usually which norm to use is determined by the convenience
of any particular analysis, but for all practical purpose, any norm is just as good as another. In our theoretical analysis below,
we use both $\|\cdot\|_2$, the matrix spectral norm, and $\|\cdot\|_{\F}$, the matrix Frobenius norm \cite{demmel1997applied}, but
use only $\|\cdot\|_{\F}$ in our numerical tests because then \eqref{eq:measure} is computable.
Additionally, in using \eqref{eq:measure}, we usually restrict $W$ and $\wtd W$ to $\mathbb{W}_{\tau_n}$.

\subsection{A Case Study: MICA}
MICA \cite{cardoso1998multidimensional,poczos2005independent,theis2006towards}
aims at separating linearly mixed unknown sources into statistically independent
groups of signals.
A basic MICA model can be stated as
\begin{align}
x=M s+v,
\end{align}
where $x\in\mathbb{R}^n$ is the observed mixture,
$M\in\mathbb{R}^{n\times n}$ is a nonsingular matrix (often called {\em the mixing matrix\/}),
$s\in\mathbb{R}^n$ is the source signal,
and $v\in\mathbb{R}^n$ is the noise vector.

We would like to recover the source $s$ from
the observed mixture $x$.
%Denote $\tau_n=(n_1,\dots, n_t)\in\mathbb{T}_n$.
Let $s = \big[s_1^{\T}, \dots , s_t^{\T} \big]^{\T}$ with $s_j\in\mathbb{R}^{n_j}$
for $j=1,2,\ldots,t$,
and $v=[\nu_1, \dots , \nu_n]^{\T}$.
Assume that all $s_j$ are independent of each other, and each
$s_j$ has mean $0$ and contains no lower-dimensional independent component,
and among all $s_j$, there exists at most one Gaussian component.
Assume further that the noises $\nu_1,\dots,\nu_n$ are real stationary white random signals,
mutually uncorrelated with the same variance $\sigma^2$,
and independent of the sources.
To recover the source signal $s$, it suffices to find $M$ or its inverse from the observed mixture $x$.
Notice that if $M$ is a solution, then so is $MD\Pi$, where $D$ is a block diagonal scaling matrix and
$\Pi$ is a block-wise permutation matrix. In this sense, there is certain degree of freedom
in the determination of $M$. Such indeterminacy of the solution is natural, and does not matter in applications.
We have the following statements.
\begin{enumerate}
\renewcommand{\labelenumi}{(\alph{enumi})}
\item  The covariance matrix
      %\footnote{Higher cumulants of $x$ can also be considered, say the kurtosis.}
       $R_{xx}$ of $x$
      satisfies
      \begin{equation}\label{rxx}
      R_{xx}=\E(xx^{\T})=M\E(ss^{\T})M^{\T}+\E(vv^{\T})=MR_{ss}M^{\T}+\sigma^2 I,
      \end{equation}
      where $\E(\,\cdot\,)$ stands for the mathematical expectation,
      and $R_{ss}$ is the covariance matrix of $s$.
      By the above assumptions,
      we know that $R_{ss}\in\mathbb{D}_{\tau_n}$.
      Assume that $\sigma$ is accurately estimated as $\hat{\sigma}$. Then we have
      \begin{equation}\label{rxx2}
      R_{xx}-\hat{\sigma}^2 I \approx MR_{ss}M^{\T}.
      \end{equation}
      In particular, in the absence of  noises, i.e., $\sigma=0$,
      \eqref{rxx2} becomes an equality.

\item  The kurtosis\footnote{Other cumulants can also be considered.}  $\mathcal{C}_x^4$ of $x$ is a tensor of dimension $n\times n\times n\times n$.
      Fixing two indices, say the first two, and varying the last two, we have
      \begin{equation}
      \mathcal{C}_x^4(i_1,i_2,:,:)=M \mathcal{C}_s^4(i_1,i_2,:,:) M^{\T},
      \end{equation}
      where $\mathcal{C}_s^4$ is the kurtosis of $s$ and it can be shown that
      $ \mathcal{C}_s^4(i_1,i_2,:,:)\in\mathbb{D}_{\tau_n}$.
\end{enumerate}
Together, they result in a \jbdp\   for
$\wtd\cA=\{R_{xx}-\hat{\sigma} I\}\cup\{ \mathcal{C}_x^4(i_1,i_2,:,:)\}_{i_1,i_2=1}^n$.
$W:=M^{-\T}$ is an exact $\tau_n$-block diagonalizer when no noise is presented. When we attempt to
block-diagonalize $\wtd\cA$,
all we can do is to calculate an approximation $\ww$ of
$M^{-\T}D\Pi$ for some $D\in\mathbb{D}_{\tau_n}$ and $\Pi\in\mathbb{P}_{\tau_n}$,
which corresponds to the indeterminacy of MICA (even in the case when $\sigma=0$, i.e., there is no noise).

The point we try to make from this case study is that, in practical applications,
due to measurement errors, we only get to work with $\wtd\cA=\{\wtd A_i\}$ that are, in general,
only approximately block diagonalizable and, in the end, an approximate block diagonalizer $\wtd W$
of $\wtd\cA$ gets computed.
In the other word, we usually don't have $\cA$ which is known
block diagonalizable in theory but what we do have is $\wtd\cA$ which may or may not be block diagonalizable and
for which we have an approximate block diagonalizer $\wtd W$.
Then how far this $\wtd W$ is from the exact diagonalizer $W$
of $\cA$
becomes a central question, in order
to gauge the quality of $\wtd W$.
This is what we set out to do in this paper. Our result is an
upper bound on the measure in \eqref{eq:measure}.
Such an upper bound will also help us understand what are the inherent factors that
affect the sensitivity of \jbdp.

\subsection{Related works}
Though tremendous efforts have gone to solve \jdp/\jbdp,
their perturbation problems had received little or no attention in the past. In fact, today there
are only a handful articles written on the perturbations of \jdp\ only.
For \ojdp, Cardoso \cite{cardoso1998multidimensional} presented a first order perturbation
bound for a set of commuting matrices, and the result was later generalized by Russo \cite{russo2011argument}.
For general \jdp, using gradient flows, Afsari \cite{afsari2008sensitivity} studied  sensitivity
via cost functions and obtained first order perturbation bounds for the diagonalizer.
Shi and Cai \cite{shi2015some} investigated a normalized \jdp\  through a constrained optimization problem,
and obtained an upper bound on
certain distance between an approximate diagonalizer of a perturbed optimization problem and
an exact diagonalizer of the unperturbed optimization problem.

\jbdp\    can also be regarded as a particular case of the {\em block term decomposition\/} (BTD)
of third order tensors \cite{de2008decompositions, de2008decompositions2, de2008decompositions3, nion2011tensor}.
The uniqueness conditions of tensor decompositions,
which is strongly connected to the sensitivity of tensor decompositions,
received much attention recently (see, e.g., \cite{de2008decompositions2,   domanov2013uniqueness,domanov2013uniqueness2, kruskal1977three, sorensen2015new, sorensen2015coupled, stegeman2011uniqueness}).
However, perturbation theory for tensor decompositions,
often referred to as {\em identifiability of tensors}, up to now, is only discussed for the so-called
{\em canonical polyadic decomposition\/} (CPD) (see \cite{vannieuwenhoven2016condition} and references therein).
Perturbation theories for other models of tensor decompositions, e.g., the Tucker decomposition and BTD,
have not been touched yet.
More work is obviously needed in this area.

\subsection{Our contribution and the organization of this paper}
A biggest reason as to why no available perturbation analysis for \jbdp\  is, perhaps,
due to lacking perfect ways to uniquely describe block diagonalizers, not to mention no
available uniqueness condition to nail  them down, unlike many other matrix perturbation problems surveyed in \cite{li:2014HLA}.
Quite recently, in the sense of Definition~\ref{dfn:us-jbdp}, Cai and Liu \cite{cai2017algebraic}
established necessary and sufficient conditions for a \jbdp\ to be uniquely block diagonalizable.
These conditions are the cornerstone for our current investigation
in this paper.
Unlike the results in existing literatures,
the result in this paper does not involve any cost function,
which makes it widely applicable to any approximate diagonalizer computed from min/maximizing a cost function.
The result also reveals the inherent factors that affect the sensitivity of  \jbdp.

The rest of this paper is organized as follows.
In section~\ref{sec:uniq}, we discuss properties of a uniquely block diagonalizable \jbdp\ and
introduce the concepts of the moduli of uniqueness and non-divisibility that play key roles in
our later development. Our main result is presented in section~\ref{sec:main}, along
with detailed discussions on its numerous implications. The proof of the main result
is rather long and technical and thus is deferred to section~\ref{sec:proof}.
We validate our theoretical contributions by numerical tests reported in section~\ref{sec:numer}.
Finally, concluding remarks are given in section~\ref{sec:conclusion}.

\smallskip
{\bf Notation}.
%Here are some more notations which will be used throughout the rest of this paper.
${\mathbb R}^{m\times n}$ is the set of all $m\times n$ real matrices and ${\mathbb R}^m={\mathbb R}^{m\times 1}$.
$I_n$ is the $n\times n$ identity matrix, and
$0_{m\times n}$  is the $m$-by-$n$ zero matrix. When their sizes are clear from the context, we may simply write $I$ and $0$.
The symbol $\otimes$ denotes the Kronecker product.
The operation $\myvec(X)$ turns a matrix $X$ into a column vector formed by the first column of $X$ followed
by its second column and then its third column and so on.
Inversely, $\reshape(x,m,n)$ turns the $mn$-by-1 vector $x$ into an $m$-by-$n$ matrix in such a way that $\reshape(\myvec(X),m,n)=X$
for any $X\in{\mathbb R}^{m\times n}$.
The spectral norm and Frobenius norm %and infinity norm
of a matrix are denoted by $\|\cdot\|_2$ and $\|\cdot\|_{\F}$, % and $\|\cdot\|_{\infty}$,
respectively. For a square matrix $A$, $\lambda(A)$ is the set of all eigenvalues of $A$, counting algebraic multiplicities.
For convenience, we will agree that any matrix $A\in{\mathbb R}^{m\times n}$ has $n$ singular values
and $\sigma_{\min}(A)$ is the smallest one among all.

%The eigenvalue and singular value sets of  are denoted by  and $\sigma(A)$, respectively.
%%The set of all orthogonal matrix is denoted by $\mathbb{O}_n$.
%We shall also adopt {\sc matlab} convention to access the entries of vectors, matrices and tensors.

%Throughout the rest of this paper, the matrix set $\cA=\{A_i\}_{i=1}^m$, where $A_i\in{\mathbb R}^{n\times n}$, and
%$\tau_n=(n_1,\dots, n_t)\in\mathbb{T}_n$ with $t\ge 2$.
%$\kappa_2(Q)=\|Q\|_2\|Q^{-1}\|_2$,

\section{Uniquely block diagonalizable \jbdp}\label{sec:uniq}

In \cite{cai2017algebraic}, a classification of {\jbdp} is proposed. Among all and besides the one
in subsection~\ref{ssec:jbd}, there is the so-called {\em general\/} \jbdp\ (\gjbdp) for $\cA$ for which
a partition $\tau_n$ is not given but instead it asks for finding a partition $\tau_n$ with the largest
cardinality such that $\cA$ is $\tau_n$-block diagonalizable and at the same time
a  $\tau_n$-block diagonalizer. Via an algebraic approach, necessary and sufficient conditions \cite[Theorem 2.5]{cai2017algebraic}
are obtained for
the uniqueness  of (equivalent) block diagonalizers of the \gjbdp\ for $\cA$. As a corollary, we have the following result.

\begin{theorem}[{\cite{cai2017algebraic}}]\label{thm:uniq}
Given partition $\tau_n$ of $n$, suppose that the \jbdp\ of $\cA=\{A_i\}_{i=1}^m$ is $\tau_n$-block
diagonalizable and $W$ is its $\tau_n$-block diagonalizer
satisfying \eqref{eq:nojbd}.
Let $\cA_j=\{A_i^{(jj)}\}_{i=1}^m$ for $j=1,2,\ldots,t$ and assume that every $\cA_j$ cannot be further block
diagonalized{\,}\footnote{%This assumption implies that $(\tau_n,W)$ is a solution to the general \jbdp\  of $\{A_i\}_{i=1}^m$.
                 For the MICA model, this assumption is equivalent to say that
                 each component $s_j$ has no lower dimensional component.},
i.e., for any partition $\tau_{n_j}$ of $n_j$ with $\card(\tau_{n_j})\ge 2$, $\cA_j$ is not $\tau_{n_j}$-block diagonalizable.
Then  the \jbdp\   of $\cA=\{A_i\}_{i=1}^m$ is uniquely $\tau_n$-block diagonalizable
if and only if the matrix
\begin{align}\label{mjk}
M_{jk}=\sum_{i=1}^m\bsmat I_{n_k}\otimes\big[(A_i^{(jj)})^{\T} A_i^{(jj)} +A_i^{(jj)} (A_i^{(jj)})^{\T}\big] &
                                 A_i^{(kk)}\otimes A_i^{(jj)} +(A_i^{(kk)})^{\T}\otimes (A_i^{(jj)})^{\T}\\
                         \quad A_i^{(kk)}\otimes A_i^{(jj)} +(A_i^{(kk)})^{\T}\otimes (A_i^{(jj)})^{\T} &
                              \quad \big[(A_i^{(kk)})^{\T} A_i^{(kk)} +A_i^{(kk)} (A_i^{(kk)})^{\T}\big] \otimes I_{n_j}
                   \esmat
\end{align}
is nonsingular for all $1\le j< k \le t$.
\end{theorem}

%\begin{remark}
%Although Theorem~2.5 in \cite{cai2017algebraic} is established for the no-orthogonal generalized \jbdp,
%we can easily see that the result holds for \jbdp,
%for both the {\sc o-jbd} and {\sc no-jbdp}.
%\end{remark}

The following subspace of $\mathbb{R}^{n\times n}$
\begin{equation}\label{na}
\na:=\big\{Z\in\mathbb{R}^{n\times n}\; :\; A_iZ-Z^{\T}A_i=0\,\,\mbox{for $1\le i\le m$}\big\}
\end{equation}
has played an important role in the proof of  \cite[Theorem 2.5]{cai2017algebraic},
and it will also contribute to our perturbation analysis later  in a big way.

Next, let us examine some fundamental properties of $Z\in\na$ with
\begin{equation}\label{eq:calA-JBD}
A_i=\diag(A_i^{(11)},\dots, A_i^{(tt)})\quad\mbox{for $1\le i\le m$}
\end{equation}
already. Any $Z\in\na$ satisfies
\begin{equation}\label{azza}
\diag(A_i^{(11)},\dots,A_i^{(tt)})Z - Z^{\T} \diag(A_i^{(11)},\dots,A_i^{(tt)})=0
\quad\mbox{for $1\le i\le m$}.
\end{equation}
Partition $Z$ conformally as $Z=[Z_{jk}]$, where $Z_{jk}\in\mathbb{R}^{n_j\times n_k}$.
Blockwise, \eqref{azza} can be rewritten as
\begin{equation}\label{azzajk}
A_i^{(jj)} Z_{jk}  - Z_{kj}^{\T} A_i^{(kk)} =0 \mbox{ for } 1\le i\le m,\,\, 1\le j,k\le t.
\end{equation}
These equations can be decoupled into
\begin{subequations}\label{eq:decouple-azzajk}
\begin{equation}\label{zjj}
A_i^{(jj)} Z_{jj}  - Z_{jj}^{\T} A_i^{(jj)} =0\quad\mbox{for $1\le i\le m$}
\end{equation}
and for $1\le j\le t$, and
\begin{equation}\label{zjk}
A_i^{(jj)} Z_{jk}  - Z_{kj}^{\T} A_i^{(kk)} =0, \quad
A_i^{(kk)} Z_{kj}  - Z_{jk}^{\T} A_i^{(jj)} =0\quad\mbox{for $1\le i\le m$}
\end{equation}
\end{subequations}
and for $1\le j< k\le t$.

Consider first \eqref{zjk}. Together they are
 equivalent to
\begin{subequations}\label{eq:Zjk}
\begin{equation}
G_{jk}\begin{bmatrix} \hm\myvec(Z_{jk})\\  -\myvec(Z_{kj}^{\T}) \end{bmatrix}=0,
\end{equation}
where
\begin{equation}\label{hatmjk}
G_{jk}=
\left[\begin{array}{lr}
       I_{n_k}\otimes A_1^{(jj)} & (A_1^{(kk)})^{\T}\otimes I_{n_j}\\
       I_{n_k}\otimes (A_1^{(jj)})^{\T} & A_1^{(kk)}\otimes I_{n_j}\\
       \qquad\vdots & \vdots\qquad\hm\\
       I_{n_k}\otimes A_m^{(jj)} & (A_m^{(kk)})^{\T}\otimes I_{n_j}\\
       I_{n_k}\otimes (A_m^{(jj)})^{\T} & A_m^{(kk)}\otimes I_{n_j}
\end{array}\right].
\end{equation}
\end{subequations}
Notice that $M_{jk}$ defined in \eqref{mjk} simply equals to $G_{jk}^{\T}G_{jk}$.
Thus, according to Theorem~\ref{thm:uniq}, $\cA$  is uniquely $\tau_n$-block diagonalizable if and only if
the smallest singular value  $\sigma_{\min}(G_{jk})>0$, provided all $\cA_j$ cannot be further block
diagonalized.

Next, we note that \eqref{zjj}  is equivalent to
\begin{subequations}\label{eq:Zjj}
\begin{equation}\label{mzjj}
G_{jj}\myvec(Z_{jj})=0,
\end{equation}
where
\begin{equation}\label{hatmjj}
G_{jj}=\begin{bmatrix} I_{n_j}\otimes A_1^{(jj)} - \big[(A_1^{(jj)})^{\T}\otimes I_{n_j}\big]\Pi_j\\
                              \vdots \\
                        I_{n_j}\otimes A_m^{(jj)} - \big[(A_m^{(jj)})^{\T}\otimes I_{n_j}\big]\Pi_j
                 \end{bmatrix},
\end{equation}
\end{subequations}
and $\Pi_j\in\mathbb{R}^{n_j^2}$ is the perfect shuffle permutation matrix \cite[Subsection 1.2.11]{van1996matrix}
that enables
$\Pi_j\myvec(Z_{jj}^{\T})=\myvec(Z_{jj})$.

\begin{theorem}\label{thm:calA-JBD}
Suppose $\cA=\{A_i\}_{i=1}^m$ is already in the \jbd\ form with respect to $\tau_n=(n_1,\ldots,n_t)$, i.e.,
$A_i$ are given by \eqref{eq:calA-JBD}. The following statements hold.
\begin{enumerate}
\renewcommand{\labelenumi}{(\alph{enumi})}
\item $G_{jj}\myvec(I_{n_j})=0$, i.e., $G_{jj}$ is rank-deficient;

\item $\cA_j$ cannot be further block diagonalized
      if and only if for any $Z_{jj}\in\mathscr{N}(\cA_j)$,
      its eigenvalues  are either a single real number or
      a single pair of two complex conjugate numbers.

\item If\/ $\dim\mathscr{N}(\cA_j)=1$ which means either $n_j=1$ or
      the second smallest singular value of $G_{jj}$ is positive,
      then $\cA_j$ cannot be further block diagonalized.
\end{enumerate}
\end{theorem}

\begin{proof}
Item (a) holds because $Z=I_{n_j}$ clearly satisfies \eqref{zjj}.

For item (b), we will prove both sufficiency and necessity by contradiction.
%if the conclusion were not true,

($\Rightarrow$) Suppose there exists a $Z_{jj}\in\mathscr{N}(\cA_j)$ such that its eigenvalues are
neither a single real number nor a single pair of two complex conjugate numbers.
Then $Z_{jj}$ can be decomposed into $Z_{jj}=W_j \diag(D_1^{(j)},D_2^{(j)})W_j^{-1}$,
where $W_j$, $D_1^{(j)}$, $D_2^{(j)}$ are all real matrices and $\lambda(D_1^{(j)})\cap\lambda(D_2^{(j)})=\emptyset$.
Then substituting the decomposition into \eqref{zjj}, we can conclude that
$W_j^{\T}A_i^{(jj)}W_j$ for $i=1,2,\ldots,m$ are all block diagonal matrices,
contradicting to that $\cA_j$ cannot be further block
diagonalized.

($\Leftarrow$)
%If for any $Z_{jj}\in\mathscr{N}(\cA_j)$,
%its eigenvalues  are either a single real number or  a single pair of two complex conjugate numbers, then $\cA_j$ cannot be further block diagonalized.
Assume, to the contrary, that $\cA_j$ can be further block diagonalized, i.e.,
there exists a nonsingular $W_j$
such that $W_j^{\T} A_i^{(jj)} W_j =\diag(B_i^{(j1)}, B_i^{(j2)})$,
where $B_i^{j1}$, $B_i^{(j2)}$ are of order $n_{j1}$ and $n_{j2}$, respectively.
Then
$$
Z_{jj}=W_j^{-1} \diag(\gamma_1 I_{n_{j1}},\gamma_2 I_{n_{j2}})W_j\in\mathscr{N}(\cA_j),
$$
where $\gamma_1$, $\gamma_2$ are arbitrary real numbers.
That is that some $Z_{jj}\in\mathscr{N}(\cA_j)$ can have distinct real eigenvalues, a contradiction.

Lastly for item (c), assume, to the contrary, that  $\cA_j$ can be further block diagonalized.
Without loss of generosity, we may assume that
there exists a nonsingular matrix $W_j\in\mathbb{R}^{n_j\times n_j}$ such that
$W_j^{\T} A_i^{(jj)} W_j=\diag(A_{i}^{(jj1)},A_i^{(jj2)})$ for $i=1,2,\ldots,m$,
where $A_i^{(jj1)}$ and $A_i^{(jj2)}$ are respectively of order $n_{j1}$ and $n_{j2}$.
Then \eqref{zjj} has at least two linearly independent solutions  $W_j\diag(I_{n_{j1}},0)W_j^{-1}$, $W_j\diag(0,I_{n_{j2}})W_j^{-1}$.
Therefore, \eqref{mzjj} has two linearly independent solutions, which implies that
the second smallest singular value of the coefficient matrix $G_{jj}$ must be $0$, a contradiction.
\end{proof}

In view of Theorems~\ref{thm:uniq} and \ref{thm:calA-JBD}, we introduce the moduli of uniqueness and non-divisibility for
$\tau_n$-block diagonalizable $\cA$.

\begin{definition}\label{def:sig}
Let $W\in\mathbb{W}_{\tau_n}$ be a $\tau_n$-block diagonalizer of $\cA=\{A_i\}_{i=1}^m$
such that \eqref{eq:nojbd} holds, and let $\cA_j=\{A_i^{(jj)}\}_{i=1}^m$ for $j=1,2,\ldots,t$.
\begin{enumerate}
\renewcommand{\labelenumi}{(\alph{enumi})}
\item The {\em modulus of uniqueness\/} of the \jbdp\  for $\cA$ with respective to
      the $\tau_n$-block diagonalizer $W$ is defined by
      \begin{equation}\label{sigu}
      \omega_{\uniq}\equiv\omega_{\uniq}(\cA; W)=\min_{1\le j<k\le t} \sigma_{\min}(G_{jk}),
      \end{equation}
      where  $G_{jk}$ is given by \eqref{hatmjk}.
\item Suppose that none of $\cA_j$ can be further block diagonalized. The {\em modulus of non-divisibility\/}
      $\omega_{\robu}\equiv\omega_{\robu}(\cA; W)$ of the \jbdp\  for $\cA$ with respective to
      the $\tau_n$-block diagonalizer $W$  is defined by
      $\omega_{\robu}=\infty$ if $\tau_n=(1,1,\ldots,1)$ and
      \begin{equation}\label{sigr}
      \omega_{\robu}
         =\min_{n_j>1 }\{\mbox{the smallest nonzero singular value of $G_{jj}$}\},
      \end{equation}
      otherwise, where  $G_{jj}$ is given by \eqref{hatmjj}.
\end{enumerate}
\end{definition}

Note the notion of the modulus of non-divisibility is defined under the condition that none of $\cA_j$ can be further block diagonalized.
It is needed because in order for
\eqref{sigr} to be well-defined, we need to make sure that $G_{jj}$ has at least one nonzero singular value
in the case when $n_j>1$. In deed, $G_{jj}\ne 0$ whenever $n_j>1$, if none of $\cA_j$ can be further block diagonalized.
To see this, we note
$G_{jj}=0$ implies that any matrix $Z_{jj}$ of order $n_j$ is a solution to \eqref{zjj} and thus
$A_i^{(jj)}$ for $1\le i\le m$ are diagonal, which means that $\cA_j$ can be
further (block) diagonalized. This contradicts to the assumption that none of $\cA_j$ can be further block diagonalized.

The corollary below partially justifies Definition~\ref{def:sig}.

\begin{corollary}\label{cor:modulus}
Let $W\in\mathbb{W}_{\tau_n}$ be a $\tau_n$-block diagonalizer of $\cA=\{A_i\}_{i=1}^m$
such that \eqref{eq:nojbd} holds, and let $\cA_j=\{A_i^{(jj)}\}_{i=1}^m$.
Suppose $\dim\mathscr{N}(\cA_j)=1$ for all $1\le j\le t$, and let $\sigma_{-2}^{(j)}$
be the second smallest singular value of $G_{jj}$ for $j=1,2,\ldots,t$ whenever $n_j>1$. Then the following statement holds.
\begin{enumerate}
\renewcommand{\labelenumi}{(\alph{enumi})}
\item  $\cA$ is uniquely $\tau_n$-block diagonalizable if $\omega_{\uniq}(\cA; W)>0$.
\item None of $\cA_j$ can be further block diagonalized and
      $$
      \omega_{\robu}\equiv\omega_{\robu}(\cA; W)=\min_{n_j>1}\sigma_{-2}^{(j)}>0.
      $$
\end{enumerate}
\end{corollary}

\begin{remark}
A few comments are in order.
\renewcommand{\labelenumi}{(\alph{enumi})}
\begin{enumerate}
\item The definition of $\omega_{\uniq}$  is a natural generation of the modulus of uniqueness in \cite{shi2015some}
      for \jdp\  (i.e., when $\tau_n=(1,1,\ldots,1)$).

\item By Theorem~\ref{thm:calA-JBD}(a), we know the smallest singular value of $G_{jj}$ is always $0$. Thus it
      seems natural that in defining $\omega_{\robu}$ in \eqref{sigr}, one would
      expect using the {\em second\/} smallest singular value of $G_{jj}$. It turns out that
      there are examples for which $\cA_j$ cannot be further block diagonalized and yet
      $\dim\mathscr{N}(\cA_j)=2$, i.e., the second smallest singular value of $G_{jj}$ is still $0$.

      Consider $A_i= \bsmat \alpha_i & \hm\beta_i \\ \beta_i & -\alpha_i\esmat$ for $i=1,2,\ldots,m$, where all
      $\alpha_i, \beta_i\ne 0\in{\mathbb R}$ and ${\alpha_i}/{\beta_i}$ are not a constant.
      Then $\cA=\{A_i\}_{i=1}^m$ cannot be simultaneously  diaognalized and
      $\mathscr{N}(\cA)=\subspan\{I_2,   \bsmat \hm 0 & 1 \\ -1 & 0\esmat\}$, i.e.,
      $\dim\mathscr{N}(\cA)=2$.

%\item If $\cA$ is uniquely $\tau_n$-block diagonalizable, then $G_{jj}\ne 0$ whenever $n_j>1$.

%\item
%      If $\dim\mathscr{N}(\cA_j)=1$, then $\sigma_{-2}^{(j)}>0$. Therefore if $\dim\mathscr{N}(\cA_j)=1$ for $j=1,2,\ldots,t$, then
%      $$
%      \omega_{\robu}=\min_{n_j>1}\sigma_{-2}^{(j)}>0.
%      $$
%      Intuitively, the larger $\omega_{\robu}$ is, the further away of $\cA_j$ from being further block-diagonalized.
\end{enumerate}
\end{remark}

The moduli $\omega_{\uniq}$ and $\omega_{\robu}$, as defined in Definition~\ref{def:sig},
depend on the choice of the diaognalizer $W$. But, as the following theorem shows,  in the case when $\cA=\{A_i\}_{i=1}^m$
is uniquely $\tau_n$-block diagonalizable, their dependency on diagonalizer $W\in\mathbb{W}_{\tau_n}$ can be removed.

\begin{theorem}\label{thm:modulus}
If $\cA=\{A_i\}_{i=1}^m$ is uniquely $\tau_n$-block diagonalizable,
then $\omega_{\uniq}$ and $\omega_{\robu}$ are both independent of the choice of
diagonalizer $W\in\mathbb{W}_{\tau_n}$.
\end{theorem}

\begin{proof}
Let $W\in\mathbb{W}_{\tau_n}$ be a $\tau_n$-block diagonalizer of $\cA$. Then all
possible $\tau_n$-block diagonalizer of $\cA$ from $\mathbb{W}_{\tau_n}$ take the form $\ww=WD\Pi$
for some $D\in\mathbb{D}_{\tau_n}$ and $\Pi\in\mathbb{P}_{\tau_n}$.
We will show that $\omega_{\uniq}(\cA; \wtd W)=\omega_{\uniq}(\cA; W)$
and $\omega_{\robu}(\cA; \wtd W)=\omega_{\robu}(\cA; W)$.

We can write $D=\diag(D_1,\dots,D_t)$, where $D_j\in{\mathbb R}^{n_j\times n_j}$.
All $D_j$ are all orthogonal since $W,\,\ww\in\mathbb{W}_{\tau_n}$.
We have
\begin{align*}
\ww^{\T} A_i\ww
&=\Pi^{\T} \diag(D_1^{\T} A_i^{(11)}D_1,\dots,D_t^{\T} A_i^{(tt)}D_t)\Pi\\
&=\diag(\Pi_1^{\T} D_{\ell_1}^{\T} A_i^{(\ell_1\ell_1)}D_{\ell_1}\Pi_1, \dots,
\Pi_t^{\T} D_{\ell_t}^{\T} A_i^{(\ell_t \ell_t)}D_{\ell_t}\Pi_t),
\end{align*}
where $\{\ell_1,\ell_2,\dots,\ell_t\}$ is a permutation of $\{1,2,\dots,t\}$, and $\Pi_j$ is a
permutation matrix of order $n_j$ for $j=1,\dots, t$. Denote by
$\wtd A_i^{(jj)}=\Pi_j^{\T} D_{\ell_j}^{\T} A_i^{(\ell_j \ell_j)}D_{\ell_j}\Pi_j$, and
define $\widetilde{G}_{jk}$, accordingly as $G_{jk}$ in \eqref{hatmjk}, but in terms of $\wtd A_i^{(jj)}$
and $\wtd A_i^{(kk)}$,
$\widetilde{G}_{jj}$, accordingly as $G_{jj}$ in \eqref{hatmjj}, but in terms of $\wtd A_i^{(jj)}$.
Then by calculations, we have
\begin{align*}
\widetilde{G}_{jk}&=\big[I_{2m}\otimes(\Pi_kD_{\ell_k})^{\T}\otimes (\Pi_jD_{\ell_j})^{\T})\big]
                       G_{jk}\big[I_{2}\otimes(\Pi_kD_{\ell_k})\otimes (\Pi_jD_{\ell_j}))\big],\\
\widetilde{G}_{jj}&=\big[I_{m}\otimes(\Pi_jD_{\ell_j})^{\T}\otimes (\Pi_jD_{\ell_j})^{\T})\big]
                       G_{jj}\big[(\Pi_kD_{\ell_k})\otimes (\Pi_jD_{\ell_j})\big],
\end{align*}
which imply that the singular values of $\widetilde{G}_{jk}$ and $\widetilde{G}_{jj}$
are the same as those of $G_{jk}$ and $G_{jj}$, respectively.
The conclusion follows.
\end{proof}

\section{Main Perturbation Results}\label{sec:main}
In this section, we present our main theorem,
along with some illustrating examples and discussions on its implications. We defer its lengthy proof to section~\ref{sec:proof}.
%To prove the main theorem, three lemmas are given in subsection~\ref{subsec:lemma}.
%In the end, the proof of the main theorem is given in subsection~\ref{subsec:proof}.

\subsection{Set up the stage}\label{ssec:stage}
In what follows, we will set up the groundwork for our perturbation analysis and explain
some of our assumptions.

As before, $\cA=\{A_i\}_{i=1}^n$ is the upperturbed matrix set, where all $A_i\in{\mathbb R}^{n\times n}$, and
$\tau_n=(n_1,\ldots,n_t)$ is a partition of $n$ with $t\ge 2$. We assume that
\begin{equation}\label{eq:cA-assume}
\framebox{
\parbox{10cm}{
$\cA$ is
$\tau_n$-block diagonalizable, $W\in\mathbb{W}_{\tau_n}$ is its $\tau_n$-block diagonalizer such that \eqref{eq:nojbd}
holds, and, moreover, $\dim\mathscr{N}(\cA_j)=1$ for all $j$,
where $\cA_j=\{A_i^{(jj)}\}_{i=1}^m$ for $1\le j\le t$.
}
}
\end{equation}
The assumption that $\dim\mathscr{N}(\cA_j)=1$ implies that $\cA_j$ cannot be further block diagonalized
by Theorem~\ref{thm:calA-JBD}(c).

Suppose that $\cA=\{A_i\}_{i=1}^n$ is perturbed to
$\wtd\cA=\{\wtd A_i\}_{i=1}^m\equiv\{A_i+\Delta A_i\}_{i=1}^m$, and let
\begin{equation}\label{eq:pert-cA}
\|\cA\|_{\F}:=\left(\sum_{i=1}^m\|A_i\|_{\F}^2\right)^{1/2}, \quad
\delta_{\cA}:=\left(\sum_{i=1}^m\|\Delta A_i\|_{\F}^2\right)^{1/2}.
\end{equation}
Previously, we commented on that, more often than not, a generic \jbdp\ may
not be $\tau_n$-block diagonalizable for $m\ge 3$.
This means that $\wtd\cA$ may not be $\tau_n$-block diagonalizable
regardless how tiny $\delta_{\cA}$ may be.
For this reason, we will not assume that $\wtd\cA$ is $\tau_n$-block diagonalizable, but, instead, it has an
approximate $\tau_n$-block diagonalizer $\wtd W\in\mathbb{W}_{\tau_n}$ in the sense that
\begin{equation}\label{eq:tcA-almost-D}
\mbox{all $\wtd W^{\T}\wtd A_i\wtd W$ are nearly $\tau_n$-block diagonal}.
\end{equation}
Doing so has two advantages.
Firstly, it serves all practical purposes well, because in any likely
practical situations we usually end up with $\wtd\cA$ which is close to
some $\tau_n$-block diagonalizable $\cA$ that is not actually available due to unavoidable noises such as in MICA,
and, at the same time, an approximate $\tau_n$-block diagonalizer can be made available by computation.
Secondly, it is general enough to cover the case
when the \jbdp\ for $\wtd\cA$ is actually $\tau_n$-block diagonalizable.

We have to quantify the statement \eqref{eq:tcA-almost-D} in order to proceed.
To this end, we pick a diagonal matrix $\Gamma=\diag(\gamma_1 I_{n_1},\dots, \gamma_t I_{n_t})$,
where $\gamma_1,\dots,\gamma_t$ are distinct real numbers
with all $|\gamma_j|\le 1$, and define the $\tau_n$-block diagonalizablility residuals
\begin{equation}\label{eq:diag-res}
\wtd R_i=\ww^{\T}\wtd A_i \ww \Gamma- \Gamma \ww^{\T} \wtd A_i \ww\quad  \mbox{for } i=1,2,\ldots,m.
\end{equation}
Notice $\Bdiag_{\tau_n}(\wtd R_i)=0$ always no matter what $\Gamma$ is.
The rationale behind defining these residuals is in the following proposition.

\begin{proposition}\label{prop:diag-res}
$\wtd W^{\T}\wtd A_i\wtd W$ is $\tau_n$-block diagonal, i.e.,
$\OffBdiag_{\tau_n}(\wtd W^{\T}\wtd A_i\wtd W)=0$ if and only if $\wtd R_i=0$.
\end{proposition}

As far as this proposition is concerned, any diagonal $\Gamma$ with distinct diagonal entries suffices.
But later, we will see that our upper bound depends on $\Gamma$, which makes us wonder
what the best $\Gamma$ is for the best possible bound. Unfortunately, this is not a trivial task
and would be an interesting subject for future studies. We will return to this later in our numerical example
section. We restrict $\gamma_i$ to real numbers for consistency consideration since $\cA$ and $\wtd\cA$
are assumed real. All developments below work equally well even if they are complex.
For later use, we set
\begin{equation}\label{eq:res-err}
g=\min_{j\ne k} |\gamma_j-\gamma_k|,\quad
\tilde r=\left(\sum_{i=1}^m\|\wtd R_i\|_{\F}^2\right)^{1/2}.
\end{equation}

In addition to Proposition~\ref{prop:diag-res}, another benefit of defining the residuals $\wtd R_i$ can be seen
through backward error analysis. In fact, all $\wtd R_i$ being nearly zeros, i.e., tiny $\tilde r$, implies that
$\wtd\cA$ is nearby an exact $\tau_n$-block diagonalizable matrix set.

\begin{proposition}\label{prop:backerr}
$\ww$ is an exact $\tau_n$-block diagonalizer of the matrix set $\{\wtd A_i+E_i\}_{i=1}^m$
with relative backward error
\begin{equation}\label{eq:backerr2}
\frac {\|\cE\|_{\F}}{\|\wtd\cA\|_{\F}}\le\frac {\|\ww^{-1}\|_2^2}{\|\wtd\cA\|_{\F}}\cdot\frac {\tilde r}g
   =:\varepsilon_{\berr}(\wtd\cA;\ww),
\end{equation}
where $\cE=\{E_i\}_{i=1}^m$ which will be referred to as
the {\em backward perturbation\/} to $\wtd\cA$ with respect to the approximate diagonalizer $\ww$.
\end{proposition}

\begin{proof}
Partition $\wtd R_i$ as $\wtd R_i=\big[\wtd R_i^{(jk)}\big]$
with $\wtd R_i^{(jk)}\in\mathbb{R}^{n_j\times n_k}$.
Then \eqref{eq:diag-res} can be rewritten as
\begin{equation}\label{backerr}
\ww^{\T} (\wtd A_i +  E_i)\ww \Gamma - \Gamma \ww^{\T} (\wtd A_i +  E_i ) \ww=0,
\end{equation}
where $ E_i=\ww^{-\T}\big[E_i^{(jk)}\big]\ww^{-1}$ with $E_i^{(jj)}=0$ and
$E_i^{(jk)}=\frac{\wtd R_i^{(jk)}}{\gamma_k-\gamma_j}$ for $j\ne k$.
Let $\cE=\{E_i\}_{i=1}^m$ which satisfies \eqref{eq:backerr2}.
\end{proof}

\subsection{Main Result}\label{subsec:main}
With the setup, we are ready to state our main result.

\begin{theorem}\label{thm:main}
Adopt the setup in {\em subsection~\ref{ssec:stage}} up to \eqref{eq:diag-res}.
%Assume \eqref{eq:cA-assume} and let $\ww\in{\mathbb W}_{\tau_n}$ be an approximate $\tau_n$-block
%diagonalizer of $\wtd\cA=\{\wtd A_i\}_{i=1}^m\equiv\{A_i+\Delta A_i\}_{i=1}^m$. Define
%Let $W\in\mathbb{W}_{\tau_n}$ be a solution to the \jbdp\  of $\cA=\{A_i\}_{i=1}^m$
%with respect to $\tau_n$
%such that \eqref{eq:nojbd} holds, and let $\cA_j=\{A_i^{(jj)}\}_{i=1}^m$.
%Assume that $\dim\mathscr{N}(\cA_j)=1$ for all $1\le j\le m$.
%Let $\cA$ be perturbed to $\mathcal{\wtd A}=\{\wtd A_i\}_{i=1}^m$, where $\wtd A_i=A_i+\Delta A_i$
%for $1\le i\le m$, and
%\begin{equation}\label{eq:pert-cA}
%\delta_{\cA}=\left(\sum_{i=1}^m\|\Delta A_i\|_{\F}^2\right)^{1/2}.
%\end{equation}
%Suppose that $\ww\in\mathbb{W}_{\tau_n}$ is an approximate solution to the \jbdp\
%of $\mathcal{\wtd A}$, which may or may not be
%solvable, and let $Q=W^{-1}\ww$. For any given
%$\Gamma=\diag(\gamma_1 I_{n_1},\dots, \gamma_t I_{n_t})$,
%where $\gamma_1,\dots,\gamma_t$ are distinct real numbers
%with all $|\gamma_j|\le 1$, define
%\begin{subequations}\label{eq:app-JBD-tcA}
%\begin{align}
%g&=\min_{j\ne k} |\gamma_j-\gamma_k|,\label{ga}\\
%\wtd R_i&=\ww^{\T}\wtd A_i \ww \Gamma- \Gamma \ww^{\T} \wtd A_i \ww,\quad  \mbox{for } i=1,2,\ldots,m,\label{wr}\\
%\tilde r&=\left(\sum_{i=1}^m\|\wtd R_i\|_{\F}^2\right)^{1/2}.\label{ra}
%\end{align}
%\end{subequations}
Let $Q=W^{-1}\ww$, and let $\omega_{\uniq}$ and $\omega_{\robu}$ be defined in {\em Definition~\ref{def:sig}},
and{}\footnote {Recall that $t\ge 2$. The quantity $\tau$ decreases as $t$ increases and thus $\tau\le\sqrt 2-1$. Since
    $\alpha$ increases as $\tau$ does, $\alpha$ decreases as $t$ increases and thus
    $\alpha\le 2(\sqrt 2-1)/(2\sqrt 2-1)^2<1/4$.}
\begin{gather}
%\tau=\frac 1{\sqrt{t-1}}, \quad \alpha=\frac {2\tau}{(\sqrt 2+\tau)^2},
\tau=\frac {\sqrt{2}-1}{\sqrt{t-1}},\quad \alpha=\frac {2\tau}{(\sqrt 2+\tau)^2}, \label{eq:tau-alpha}\\
\delta=\|Q^{-1}\|_2^2\,\tilde r+2\|Q^{-1}\|_2\|W\|_2\|\ww\|_2\, \delta_{\cA},
\quad \epsilon_*=\frac{\tau \kappa_2(Q)\delta}{\alpha g\,\omega_{\uniq}}. \label{eq:del}
\end{gather}
If
\begin{equation}\label{eq:pert-tiny-cond}
\delta <\min\left\{ \frac{\alpha g\,\omega_{\uniq}}{\kappa_2(Q)},\; \frac{(1-2\alpha)g\,\omega_{\robu}}{\sqrt{2}}\right\},
\end{equation}
then for ${\scriptstyle\p}\in\{2,{\scriptstyle\F}\}$
\begin{align}
\min_{\substack{D\in\mathbb{D}_{\tau_n}, D^{\T}D=I\\ \Pi\in\mathbb{P}_{\tau_n}}}\frac{\|W - \ww D\Pi\|_{\p}}{\|\ww\|_{\p}}
   &\le \frac {1+\sqrt{t}\,\epsilon_*}
              {\sqrt{1-2\sqrt{t-1}\epsilon_* -(t-1)\epsilon_*^2}} - 1 \label{ineq:main}\\
&=\frac{\tau}{\alpha}\cdot\frac{(\sqrt{t}+\sqrt{t-1})\kappa_2(Q)\delta }{g\,\omega_{\uniq}}+O(\delta^2) \nonumber
:=\varepsilon_{\ub}. \label{ineq:main}
\end{align}
%where ${\scriptstyle\p}\in\{2,{\scriptstyle\F}\}$, and
%$$
%\epsilon_*=%\frac{\tau}{\alpha}\cdot
%\frac{\tau \kappa_2(Q)\delta}{\alpha g\,\omega_{\uniq}}.
%$$
\end{theorem}

In what follows, we first look at two illustrating examples, then discuss the implications of Theorem~\ref{thm:main}.
\begin{example}
Let  $A_1=I_2$, $A_2=\diag(1,1+\varsigma)$,
where $\varsigma>0$ is a parameter. It is obvious that $W=I_2$ is a diagonalizer of $\cA=\{A_1,A_2\}$
with respect to $\tau_2=(1,1)$. By calculations, we  get
\[
\omega_{\uniq}=\sqrt{\varsigma^2+2\varsigma+4 - (\varsigma+2)\sqrt{\varsigma^2+4}}=\frac{\varsigma}{\sqrt{2}}+O(\varsigma^{3/2}), \quad
\omega_{\robu}=\infty.
\]
Perturb $\cA$ to $\wtd\cA=\{\wtd A_1,\wtd A_2\}$, where $\wtd A_1=A_1+\epsilon E$ and
$\wtd A_2=A_2-\epsilon E$, with
$E=\begin{bmatrix} \hm 1 & 1\\ -1 & 1\end{bmatrix}$, and $\epsilon\ge 0$ is a parameter
for controlling the level of perturbation.
Consider
$$
c=\cos\theta,\quad s=\sin\theta, \quad
\ww=\begin{bmatrix} \hm c &  s \\
          -s  & c
    \end{bmatrix},
$$
where $\theta\in [-\frac{\pi}{2},\frac{\pi}{2}]$ is a parameter that controls the quality of approximate diagonalizer
$\wtd W$ of $\wtd\cA$.
Simple calculations give
$$
\ww^{\T}\wtd A_1\ww=\begin{bmatrix} 1+\epsilon & \epsilon \\
                                  -\epsilon & 1+ \epsilon
                            \end{bmatrix}, \quad
\ww^{\T} \wtd A_2\ww=\begin{bmatrix} 1+\varsigma s^2-\epsilon & -\epsilon-\varsigma cs \\
                                   \epsilon - \varsigma cs & 1+\varsigma c^2-\epsilon
                            \end{bmatrix}
$$
from which we can see that if $\theta$ and $\epsilon$ are sufficiently small,
$\ww$ is a good block diagonalizer.
Now let $\Gamma=\diag(-1,1)$. We have
$$
g=2,\quad
\kappa_2(Q)=1,\quad
\tilde r=\sqrt{16\epsilon^2+8\varsigma^2c^2s^2},\quad
\delta_{\cA}=2\sqrt{2} \epsilon,\quad
\delta=\tilde r+2\delta_{\cA}.
$$
Thus, if $\theta=\epsilon$ and $\epsilon\ll 1$, then \eqref{eq:pert-tiny-cond} is satisfied.
Thus, by \eqref{ineq:main}, for ${\scriptstyle\p}\in\{2,{\scriptstyle\F}\}$
$$
\min_{D,\Pi}\frac{\|W-\ww D \Pi\|_{\p}}{\|\ww\|_{\p}}=2\sin\frac{\theta}{2}\approx \epsilon,\quad
\varepsilon_{\ub}\approx\frac{(1+5\sqrt{2})(\sqrt{16+8\varsigma^2} +4\sqrt{2})\epsilon}{4\omega_{\uniq}}.
$$
Therefore, as long as $\varsigma$ is not too small, $\omega_{\uniq}$ is not small, and then $\varepsilon_{\ub}=O(\epsilon)$,
i.e., the relative error in $\ww$ and the upper bound $\varepsilon_{\ub}$ have the same order of magnitude.
However, if $\epsilon\ll 1$ and $\varsigma$ is small,
say $\varsigma=\epsilon^{\phi}$ with $0<\phi<1$,
then $\ww$ is always a good block diagonalizer,
independent of $\theta$, in the sense that $\tilde r$ is always small.
But now we have  $\varepsilon_{\ub}=O(\epsilon^{1-\phi})$,
which does not provide a sharp upper bound for the relative error in $\ww$.
\end{example}

\begin{example}
Let $A_1=\diag(I_2, \begin{bmatrix} 1 & 1+\varsigma \\  1 & 1\end{bmatrix})$,
$A_2=\diag(\begin{bmatrix} 1 & 1+\varsigma \\  1 & 1\end{bmatrix}, I_2)$,
where $\varsigma>0$ is a parameter.
Then $W=I_4$ is a $\tau_4$-block diagonalizer of $\cA=\{A_1, A_2\}$, where $\tau_4=(2,2)$.
By calculations, we have
\[
\omega_{\uniq}\approx 0.5858+O(\varsigma),\quad
\omega_{\robu}=\varsigma.
\]
Perturb $\cA$ to $\wtd\cA=\{\wtd A_1,\wtd A_2\}$, where $\wtd A_1=A_1+\epsilon E$,
$\wtd A_2=A_2-\epsilon E$,
where $E$ is a 4-by-4 matrix of all ones and $\epsilon\ge 0$.
Consider
$$
U=\diag\Big(\frac{1}{\sqrt{2}}\begin{bmatrix} \hm 1 & 1\\ -1 & 1\end{bmatrix},
        \frac{1}{\sqrt{2}}\begin{bmatrix} \hm 1 & 1\\ -1 & 1\end{bmatrix}\Big),\quad
\ww=U \diag\Big(1,\begin{bmatrix} \hm c &  s \\
                    -s  & c\end{bmatrix},1\Big),
$$
where $c=\cos\theta$, $s=\sin\theta$, and $\theta\in [-\frac{\pi}{2},\frac{\pi}{2}]$.
Then
$$
\sum_{i=1}^2\big\|\OffBdiag_{\tau_n}(\ww^{\T}\wtd A_i\ww)\big\|_{\F}^2
 =4s^2c^2(2+\varsigma)^2+4\varsigma^2s^2+16(1+s^2)c^2\epsilon^2.
$$
Therefore, if $\theta$ and $\epsilon$ are sufficiently small,
then $\ww$ is a good block diagonalizer.
Now let $\Gamma=\diag(-I_2,I_2)$. By simple calculations, we get
\begin{gather*}
g=2,\quad
\kappa_2(Q)=1,\quad \delta_{\cA}=4\sqrt{2} \epsilon,\quad \delta=\tilde r+2\delta_{\cA},\\
\tilde r=2\sqrt{4s^2c^2(2+\varsigma)^2+4\varsigma^2s^2+16(1+s^2)c^2\epsilon^2}.
\end{gather*}
If $\theta=\epsilon\ll 1$ and $\varsigma$ is not too small, then \eqref{eq:pert-tiny-cond} is satisfied.
Thus, by \eqref{ineq:main}, for ${\scriptstyle\p}\in\{2,{\scriptstyle\F}\}$
$$
\min_{D,\Pi}\frac{\|W-\ww D \Pi\|_{\p}}{\|\ww\|_{\p}}=2\sin\frac{\theta}{2}\approx \epsilon,\quad
\varepsilon_{\ub}\approx \frac{(1+5\sqrt{2})\delta}{4\omega_{\uniq}}=O(\epsilon),
$$
i.e.,  the relative error in $\ww$ and the upper bound $\varepsilon_{\ub}$ have the same order of magnitude.
However, if $\theta=\frac{\pi}{2}-\epsilon$ with $\epsilon\ll1$ and $\varsigma$ is small,
say $\varsigma=\epsilon^{\phi}$ with $\phi>0$,
then the condition \eqref{eq:pert-tiny-cond} of Theorem~\ref{thm:main} is likely violated, and consequently,
Theorem~\ref{thm:main} is no longer applicable.
\end{example}

%\medskip

From these two examples, we can see that the bound $\varepsilon_{\ub}$ in \eqref{ineq:main} is {\em sharp}
in the sense that it can be in the same order of magnitude as the relative error.
But when  $\omega_{\uniq}$ and/or $\omega_{\robu}$ is small, Theorem~\ref{thm:main} may not provide a sharp
bound or even fails to give a bound. This observation is more or less expected. In fact, when
$\omega_{\uniq}$ and/or $\omega_{\robu}$ is small, the \jbdp\ for $\cA$ can be thought of
as an ill-conditioned problem in the sense that any small perturbation can result in huge
error in the solution.

When solving an \ojbdp, diagonalizers $W$, $\ww$ are orthogonal, and thus
$\delta =\tilde r+2\delta_{\cA}$. Theorem~\ref{thm:main} yields

\begin{corollary}\label{cor:main-ojbdp}
In {\em Theorem~\ref{thm:main}}, if  $W$ and $\ww$ are assumed orthogonal, then
\begin{equation}\label{ineq:main-ojbdp}
\min_{\substack{D\in\mathbb{D}_{\tau_n}, D^{\T}D=I\\ \Pi\in\mathbb{P}_{\tau_n}}}\frac{\|W - \ww D\Pi\|_{\p}}{\|\ww\|_{\p}}
   \le  \frac{\tau}{\alpha}\cdot\frac{(\sqrt{t}+\sqrt{t-1})\delta}{g\,\omega_{\uniq}}+O(\delta^2).
\end{equation}
\end{corollary}

Some of the quantities in the right-hand side of \eqref{ineq:main} are not computable,
unless $W$ is known. But it can still be useful in  assessing roughly how good the
approximate bock diagonalizer $\ww$ may be. Suppose that $\tilde r$ is sufficiently tiny. Then it is plausible
to assume $\|Q^{-1}\|_2=O(1)$. The moduli $\omega_{\uniq}$ and $\omega_{\robu}$ which are intrinsic to the
\jbdp\ for $\cA$ may well be estimated by those of
$\what\cA=\big\{\Bdiag_{\tau_n}(\ww^{\T}\wtd A\ww)\big\}_{i=1}^m$. Finally, for $W\in{\mathbb W}_{\tau_n}$
\begin{equation}\label{eq:W-norm}
1\le\|W\|_2\le\sqrt t.
\end{equation}
The same holds for $\ww$, too. We will justify \eqref{eq:W-norm} after Lemma~\ref{lem3} in section~\ref{sec:proof}  in order to use
some of the techniques arising in its proof.

\begin{remark}
Several comments are in order.
\begin{enumerate}
\renewcommand{\labelenumi}{(\alph{enumi})}
\item The quantity $\delta$ in \eqref{eq:del} consists of two parts:
      the first part indicates how good $\ww$ is in approximately block-diagonalizing $\wtd\cA$, and
      the second part indicates how large the perturbation is.
      Therefore, the condition \eqref{eq:pert-tiny-cond} means that the block diagonalizer $\ww$ has to be sufficiently good
      and the perturbation has to be sufficiently small so that $\delta$ does not exceed
      the right-hand side of \eqref{eq:pert-tiny-cond}, which is proportional to the moduli $\omega_{\uniq}$ and $\omega_{\robu}$.
      Although the modulus of non-divisibility $\omega_{\robu}$ does not appear explicitly in the upper bound, it
      limits the size of $\delta$.

%\item At first glance, $\varepsilon_{\ub}$ has nothing to do with $\omega_{\robu}$.
%However, $\varepsilon_{\ub}$ is in fact implicitly related with $\omega_{\robu}$,
%since $\varepsilon_{\ub}$ is function of $\delta$ and $\delta$ is controlled by $\omega_{\robu}$ according to \eqref{eq:pert-tiny-cond}.

\item In \eqref{ineq:main}, $\varepsilon_{\ub}$ is a monotonically increasing function in $\delta$ and $\kappa_2(Q)$.
      If $W$ (or $\ww$) is ill-conditioned,
      then both $\delta$ and $\kappa_2(Q)$ can be large,
      as a result, $\varepsilon_{\ub}$ can be large.

%\item The diagonal matrix $\Gamma$ is auxiliary. It serves the purpose to create a measure in terms of $\{\wtd R_i\}$
%to indicate how close the \jbdp\ for $\wtd\cA$ is solvable and the same time $\ww$ is a diagonalizer. In fact, it
%is not hard to see that the \jbdp\ for $\wtd\cA$ is solvable and is solved by $\ww$ if and only if
%all $\wtd R_i=0$ for any any $\Gamma=\diag(\gamma_1 I_{n_1},\dots, \gamma_t I_{n_t})$ with distinct $\gamma_i$. In view of this,
%$\Gamma$ is a free matrix-parameter that can be used to minimize the upper bound $\varepsilon_{\ub}$.
%Unfortunately, this is not a trivial task and will be subject future studies.
%We can use randomly generated $\Gamma$ to give an upper bound for $\varepsilon_{\ub}$.
%%$\frac{\|W - \ww D\Pi\|_{\p}}{\|\ww\|_{\p}}$.
%In our later numerical examples in section~\ref{sec:numer},
%we randomly generate many $\Gamma$,
%and compute the corresponding upper bounds for $\varepsilon_{\ub}$ and pick the best one.

\item  If $\delta \ll 1$, by \eqref{ineq:main}, we have
      \begin{equation}\label{approxmain}
      \min_{D,\Pi}\frac{\|W - \ww D\Pi\|_{\p}}{\|\ww\|_{\p}}
                 \le  \frac{\tau}{\alpha}\cdot\frac{(\sqrt{t}+\sqrt{t-1})\kappa_2(Q)}{\omega_{\uniq}} \cdot\frac{\delta}{g}
                      +O(\delta^2).
      \end{equation}

\item A natural assumption when performing a perturbation analysis
      for \jbdp\ is to assume that both the original matrix set $\cA$ and its perturbed one $\wtd\cA$
      admit exact block diagonalizers, i.e., both \jbdp are solvable.
      Theorem~\ref{thm:main} covers such a scenario as a special case with $\tilde r=0$.

%\item (Error bound)
%      A possible application of Theorem~\ref{thm:main} is to the case when $\cA$ is $\tau_n$-block diagonalizable
%      and an approximate $\tau_n$-block diagonalizer $\ww$ is available. The question is how far $\ww$ is
%      from an exact but unknown $\tau_n$-block diagonalizer of $\cA$. It can be quickly answered by
%      letting $\delta_{\cA}=0$ in Theorem~\ref{thm:main}, i.e., all $\wtd A_i=A_i$, and
%      then \eqref{ineq:main} becomes an {\em error bound\/} for $\wtd W$.

\end{enumerate}
\end{remark}

Theorem~\ref{thm:main}, as a perturbation theorem for \jbdp, can be used to yield
an error bound for an approximate block diagonalizer of block diagonalizable $\cA$ by simply letting
all $\wtd A_i=A_i$, i.e., $\delta_{\cA}=0$. In fact, when $\delta_{\cA}=0$, $\delta=\|Q^{-1}\|_2^2\, \tilde r$. If also $\tilde r\ll 1$, then $\delta \ll 1$ and thus
by  \eqref{approxmain}
\begin{equation}\label{errbound}
\min_{D,\Pi}\frac{\|W - \ww D\Pi\|_{\p}}{\|\ww\|_{\p}}
\le\frac{\tau}{\alpha}\cdot\frac{(\sqrt{t}+\sqrt{t-1})\kappa_2(Q)\|Q^{-1}\|_2^2}{\omega_{\uniq}}
        \cdot\frac{\tilde r}{g}+O({\tilde r}^2).
\end{equation}
This error bound is $O(\frac{\tilde{r}}{\omega_{\uniq}})$,
which is in agreement with the error bound when applied to \jdp\ in \cite[Corollary 3.2]{shi2015some}.
%This bound by \eqref{errbound} is unfortunately not computable

\subsection{Condition Number}
A widely accepted way to define condition number is through some kind of first order
expansion. To explain the idea, we use the explanation in \cite[p.4]{demmel1997applied} for a
real-valued differentiable function $f(x)$ of real variable $x$. Now if $x$ is perturbed to $x+\delta x$, we have,
to the first order,
$$
\frac {|f(x+\delta x)-f(x)|}{|f(x)|}\approx \frac {|f'(x)|\cdot|x|}{|f(x)|}\cdot\frac {|\delta x|}{|x|}.
$$
In words, this says that the relative change to the function value $f(x)$ is about the relative change to the
input $x$ magnified by the factor $|f'(x)|\cdot|x|/|f(x)|$ which defines the (relative) condition number
of $f(x)$ at $x$. A prerequisite for this line of definition is that $f$ is well-defined in some neighborhood of $x$.

In generalizing this framework to more broad content. The above scalar-valued function $f$ is translated into some mapping
that maps inputs which are usually much more general than a single scalar $x$
to some output. In the context of \jbdp, naturally the input is the matrix set $\cA$ and the output is
the block diagonalizer $W$. But then the framework does not work because any generic and arbitrarily small perturbation
to $\cA$ will render one that is not $\tau_n$-block diagonalizable, i.e., the mapping that takes in $\cA$ is not well-defined
in any neighborhood of $\cA$.

We have to seek some other way. Recall the rule of thumb:
$$
\mbox{forward error  $\lesssim$ condition number  $\times$ backward error}.
$$
We will use this as a guideline. Consider $\cA$ and $\wtd\cA$ which is some tiny perturbation away from $\cA$ and suppose
both are $\tau_n$-block diagonalizable with $\tau_n$-block diagonalizer $W$ and $\wtd W$ from ${\mathbb W}_{\tau_n}$, respectively. Apply
Theorem~\ref{thm:main} with $\tilde r=0$ and sufficiently tiny $\delta_{\cA}$ to get,
up to the first order in $\delta_{\cA}$,
\begin{multline*}
\min_{\substack{D\in\mathbb{D}_{\tau_n}, D^{\T}D=I\\ \Pi\in\mathbb{P}_{\tau_n}}}\frac{\|W - \ww D\Pi\|_{\p}}{\|\ww\|_{\p}} \\
  \lesssim  \frac{\tau}{\alpha}\cdot\frac{(\sqrt{t}+\sqrt{t-1})\kappa_2(Q)\|Q^{-1}\|_2\|W\|_2\|\ww\|_2\|\cA\|_{\F}}{g\,\omega_{\uniq}}
            \cdot\frac { \delta_{\cA}}{\|\cA\|_{\F}}.
\end{multline*}
%\begin{align*}
%\min_{\substack{D\in\mathbb{D}_{\tau_n}, D^{\T}D=I\\ \Pi\in\mathbb{P}_{\tau_n}}}\frac{\|W - \ww D\Pi\|_{\p}}{\|\ww\|_{\p}}
%  \approx  \frac{\tau}{\alpha}\cdot\frac{(\sqrt{t}+\sqrt{t-1})\kappa_2(Q)\|Q^{-1}\|_2\|W\|_2\|\ww\|_2\|\cA\|_{\F}}{g\,\omega_{\uniq}}\cdot\frac { \delta_{\cA}}{\|\cA\|_{\F}}
%     %+O(\delta_{\cA}^2).
%\end{align*}
Thinking about as $\wtd\cA$ goes to $\cA$,  we may let $\wtd W$ go to $W$ and the right-hand side approaches to
$$
 \frac{\tau}{\alpha}\cdot\frac{(\sqrt{t}+\sqrt{t-1})\|W\|_2^2\|\cA\|_{\F}}{g\,\omega_{\uniq}}\cdot\frac { \delta_{\cA}}{\|\cA\|_{\F}}.
    % +O(\delta_{\cA}^2)
$$
which suggests that
we may define the {\em $\tau_n$-condition number\/} of \jbdp\  for $\cA$  as
\begin{equation}\label{condnum}
\cond(\cA)= \frac{\tau}{\alpha}\cdot\frac{(\sqrt{t}+\sqrt{t-1})\|W\|_2^2 \|\cA\|_{\F}}{\omega_{\uniq}},
\end{equation}
where the notational dependency on $\tau_n$ is suppressed for convenience.
A few remarks are in order for this condition number $\cond(\cA)$.
\begin{enumerate}
\renewcommand{\labelenumi}{(\alph{enumi})}
\item As it appears, the right-hand side of \eqref{condnum} depends on the $\tau_n$-block
      diagonalizer $W\in{\mathbb W}_{\tau_n}$. But it isn't. This is because
      $\omega_{\uniq}$ is  independent of the choice of the block diagonalizer $W\in\mathbb{W}_{\tau_n}$ (Theorem~\ref{thm:modulus})
      and so is $\|W\|_2$ (Lemma~\ref{lm:2Ws} below).

%\item For a given matrix set $\cA$ and its exact block diagonalizer $W$,
%      the more the number of diagonal blocks and the larger $\|W\|_2$ is,
%      the larger the condition number becomes.

\item Given $\beta\ne 0$, let $\beta \cA=\{\beta A_i\}_{i=1}^m$.
      It can be seen that $\cond(\cA)=\cond(\beta\cA)$,
      i.e., the condition number $\cond(\cA)$ is scalar-scaling invariant.

\item Suppose $\|A_i\|_{\F}=1$ for $i=1,2,\ldots,m$ and
      consider the condition number $\cond(\what\cA)$ of the \jbdp\  for $\what\cA=\{\beta_i A_i\}_{i=1}^m$,
      where $\beta_j$ are positive real numbers.
      Recall the definition of $G_{jk}$ in \eqref{hatmjk}
      and the definition of $\omega_{\uniq}$. $W$, as a $\tau_n$-block diagonalizer of $\cA$,
      is also one of $\what\cA$.
      Now define $\what G_{jk}$ for $\what\cA$, similarly to $G_{jk}$ for $\cA$.
      We have
      \begin{equation}
      \what G_{jk}=\big[\diag(\beta_1,\dots,\beta_m)\otimes I_{2 n_j n_k}\big]G_{jk}.
      \end{equation}
      Let $\beta_{\max}=\max_{1\le j\le t} \beta_j$ and $\beta_{\min}=\min_{1\le j\le t}\beta_j$.
      We have $\sigma_{\min}(\what G_{jk})\ge \beta_{\min} \sigma_{\min}(G_{jk})$.
%      The smallest singular value $\hat{\sigma}_{jk}$ of $\what G_{jk}$
%      satisfies $\hat{\sigma}_{jk}\ge \beta_{\min} \sigma_{jk}$.
      Thus, $\hat{\omega}_{\uniq}:=\omega_{\uniq}(\what\cA)\ge \beta_{\min} \omega_{\uniq}$.
      Therefore
      \begin{equation}\label{scalea}
      \cond(\what\cA)
         =\frac{\tau}{\alpha}\cdot\frac{(\sqrt{t}+\sqrt{t-1})\|W\|_2^2\left(\sum_{i=1}^m \|\beta_iA_i\|_{\F}^2\right)^{1/2}}
                                       {\hat{\omega}_{\uniq}}
         \le \frac{\beta_{\max}}{\beta_{\min}}\cond(\cA).
      \end{equation}
      As an upper bound of $\cond(\what\cA)$,
      the right hand side of \eqref{scalea} is minimized if all $\beta_j$ are equal.
      This tells us that when solving \jbdp,
      it would be a good idea to first normalize all $A_i$ to have $\|A_i\|_{\F}=1$.

\item It is easy to see that
      the modulus of uniqueness $\omega_{\uniq}$ is an monotonic increasing function
      of the number of matrices in $\cA$. How it affects
      the condition number $\cond(\cA)$ is in general unclear.
      In our numerical tests in section~\ref{sec:numer},
      as we put more matrices into the matrix set $\cA$,
      the condition number $\cond(\cA)$ first decreases then remains almost unchanged.

\item Compared with the condition number $\mbox{cond}_{\lambda}$ introduced in \cite{shi2015some} for \jdp\  only,
      our condition number here is about the square root of $\mbox{cond}_{\lambda}$ there, and thus more realistic.
\end{enumerate}

\begin{lemma}\label{lm:2Ws}
For any two  $W,\,\wtd W\in\mathbb{W}_{\tau_n}$, if $\wtd W=WD\Pi$
for some $D\in\mathbb{D}_{\tau_n}$ and $\Pi\in\mathbb{P}_{\tau_n}$, then
$D$ is orthogonal and, as a result, $\|\wtd W\|_2=\|W\|_2$.
\end{lemma}

\begin{proof}
%By Definition~\ref{dfn:us-jbdp}, we know that there exist a $D\in\mathbb{D}_{\tau_n}$ and a $\Pi\in\mathbb{P}_{\tau_n}$ such that
%$\wtd W=WD\Pi$. We have to prove that $D$ is orthogonal, i.e., $D^{\T}D=I_n$.
Since $D\in\mathbb{D}_{\tau_n}$, $D=\diag(D_1,D_2,\ldots,D_t)$ with $D_j\in{\mathbb R}^{n_j\times n_j}$.
It suffices to show each $D_j$ is orthogonal.
Write $W=[W_1,W_2,\ldots,W_t]$ and $\wtd W=[\wtd W_1,\wtd W_2,\ldots,\wtd W_t]$, where
$W_j,\,\wtd W_j\in{\mathbb R}^{n\times n_j}$.
Because $W,\,\wtd W\in\mathbb{W}_{\tau_n}$ by assumption, we have
$$
 \Bdiag_{\tau_n}( W^{\T} W)=\Bdiag_{\tau_n}(\wtd W^{\T}\wtd W)=I_n.
$$
Because $\Pi\in\mathbb{P}_{\tau_n}$, the diagonal blocks $\{I_{n_j}\}_{j=1}^t$ of $\Bdiag_{\tau_n}(\Pi^{\T}D^{\T} W^{\T} WD\Pi)$
are the same as those of $\Bdiag_{\tau_n}(D^{\T} W^{\T} WD)$ after some permutation. Therefore,
$$
I_{n_j}=D_j^{\T}W_j^{\T}W_jD_j=D_j^{\T}D_j,
$$
i.e., $D_j$ is orthogonal for all $j$, as expected.
\end{proof}

Thus, if  \jbdp\   is uniquely $\tau_n$-block diagonalizable,
then all $\tau_n$-block diagonalizers in $\mathbb{W}_{\tau_n}$ can be written in the form $WD\Pi$,
where $W\in\mathbb{W}_{\tau_n}$ is a particular $\tau_n$-block diagonalizer,
$D\in\mathbb{D}_{\tau_n}$ is orthogonal and $\Pi\in\mathbb{P}_{\tau_n}$.

\section{Proof of Theorem~\ref{thm:main}}\label{sec:proof}
Recall the assumptions: $\cA=\{A_i\}_{i=1}^m$ is $\tau_n$-block diagonalizable and
$W\in\mathbb{W}_{\tau_n}$ is a $\tau_n$-block diagonalizer such that \eqref{eq:nojbd} holds.
The modulus of uniqueness $\omega_{\uniq}$ and the modulus of non-divisibility $\omega_{\robu}$
for the block diagonalization of $\cA$ by $W$ are defined by Definition~\ref{def:sig}.
The perturbed matrix set is $\wtd\cA=\{\wtd A_i\}_{i=1}^m$ and $\wtd W$ is an approximate
$\tau_n$-block diagonalizer of $\wtd\cA$. $\Gamma=\diag(\gamma_1 I_{n_1},\dots, \gamma_t I_{n_t})$,
where $\gamma_1,\dots,\gamma_t$ are distinct real numbers
with all $|\gamma_j|\le 1$, and $\wtd R_i$ are defined by \eqref{eq:diag-res}.

\subsection{Three Lemmas}\label{subsec:lemma}
The three lemmas in this subsection may have interest of their own, although their roles here are to
assist the proof of Theorem~\ref{thm:main}.

\begin{lemma}\label{lem1}
For given $Z\in{\mathbb R}^{n\times n}$, denote by
\begin{equation}\label{eq:Ri}
R_i=\diag(A_i^{(11)},\dots, A_i^{(tt)}) Z - Z^{\T} \diag(A_i^{(11)},\dots,A_i^{(tt)})
\end{equation}
for $1\le i\le m$.
Partition $Z=\big[Z_{jk}\big]$ with $Z_{jk}\in\mathbb{R}^{n_j\times n_k}$ and let $\lambda(Z_{jj})=\{\mu_{jk}\}_{k=1}^{n_j}$.
\begin{enumerate}
\renewcommand{\labelenumi}{(\alph{enumi})}
  \item If $\omega_{\uniq}>0$,
        then
        \begin{equation}\label{offz}
        \|\OffBdiag_{\tau_n}(Z)\|_{\F}^2  \le \frac{\sum_{i=1}^m\|\OffBdiag_{\tau_n}(R_i)\|_{\F}^2}{\omega_{\uniq}^2}.
        \end{equation}

  \item If $\dim\mathscr{N}(\cA_j)=1$,
        then there exists a real number $\hat{\mu}_j$ such that
        \begin{equation}\label{disteig}
        \sum_{k=1}^{n_j}|\mu_{jk} -\hat{\mu}_j|^2\le
        \frac{\sum_{i=1}^m\|\Bdiag_{\tau_n}(R_i)\|_{\F}^2}{\omega_{\robu}^2} .
        \end{equation}
\end{enumerate}

\end{lemma}

\begin{proof}
Partition $R_i=\big[R_i^{(jk)}\big]$ conformally with respect to $\tau_n$.
First, we show \eqref{offz}.
For any pair $(j,k)$ with $j<k$, it follows from \eqref{eq:Ri} that
\begin{align*}
G_{jk}\begin{bmatrix} \hm \myvec (Z_{jk})\\ -\myvec (Z_{kj}^{\T})\end{bmatrix}
     =\begin{bmatrix} \myvec (R_1^{(jk)})\\ -\myvec ((R_1^{(kj)})^{\T})\\
             \vdots\\
             \myvec (R_m^{(jk)})\\ -\myvec ((R_m^{(kj)})^{\T})
             \end{bmatrix}=:r_{jk},
\end{align*}
where $G_{jk}$ is defined by \eqref{hatmjk}. Put them all together to get
\[
M_{\uniq} z_{\uniq} =r_{\uniq},
\]
where
\begin{align*}
M_{\uniq}&=\diag\big(G_{12},\dots,G_{1t},G_{23},\dots,G_{2t}, \dots, G_{t-1,t}\big),\\
z_{\uniq}&=\big[\myvec (Z_{12})^{\T}, \, -\myvec (Z_{21}^{\T})^{\T},
      \dots,
      \myvec (Z_{1t})^{\T}, -\myvec (Z_{t1}^{\T})^{\T}, \\
     &\qquad \myvec (Z_{23})^{\T}, -\myvec (Z_{32}^{\T})^{\T},\dots,
     \myvec (Z_{2t})^{\T}, -\myvec (Z_{t2}^{\T})^{\T},
      \dots, \\
      &\qquad \myvec(Z_{t-1,t})^{\T}, \myvec(Z_{t,t-1}^{\T})^{\T}\big]^{\T},\\
r_{\uniq}&=\big[
        r_{12}^{\T},  \dots,  r_{1t}^{\T},  r_{23}^{\T},  \dots,  r_{2t}^{\T}, \dots,  r_{t-1,t}^{\T}
     \big]^{\T}.
\end{align*}
We have $\sigma_{\min}(M_{\uniq})=\min_{j<k}\sigma_{\min}(G_{jk})=\omega_{\uniq}>0$, and thus
\begin{align*}
\|\OffBdiag_{\tau_n}(Z)\|_{\F}^2=\|z_{\uniq}\|_2^2\le \frac{\|r_{\uniq}\|_2^2}{\omega_{\uniq}^2}
=\frac{\sum_{i=1}^m\|\OffBdiag_{\tau_n}(R_i)\|_{\F}^2}{\omega_{\uniq}^2},
\end{align*}
as expected.
Next, we show \eqref{disteig}.
For $j=k$, using \eqref{eq:Ri}, we have
\begin{align*}
G_{jj} \myvec(Z_{jj})=\begin{bmatrix} \myvec(R_1^{(jj)})\\ \vdots \\
                                       \myvec(R_m^{(jj)})
                                \end{bmatrix}=:r_{jj},
\end{align*}
where $G_{jj}$ is defined by \eqref{hatmjj}.
Since $\dim\mathscr{N}(\cA_j)=1$ by assumption, we know that the null space of $G_{jj}$
is spanned by $\myvec(I_{n_j})$, and thus
there exists a real number $\hat{\mu}_j$ such that
\[
\myvec(Z_{jj})=G_{jj}^{\dagger} r_{jj} + \hat{\mu}_j \myvec(I_{n_{j}}),
\]
where $G_{jj}^{\dagger}$ is the Moore-Penrose inverse \cite[p.102]{stewart1990matrix} of $G_{jj}$.
It follows immediately that
\begin{align*}
Z_{jj}=\what Z_{jj}+\hat{\mu}_j I_{n_j},
\end{align*}
where
$\what Z_{jj}=\reshape(G_{jj}^{\dagger} r_{jj},n_j,n_j)$. In particular,
$\lambda(\what Z_{jj})=\{\mu_{jk}-\hat{\mu}_j\}_{k=1}^{n_j}$ and hence
\begin{align*}
\sum_{k=1}^{n_j}|\mu_{jk} - \hat{\mu}_j|^2\le \|\what Z_{jj}\|_{\F}^2\le
\frac{\|r_{jj}\|_2^2}{\omega_{\robu}^2}
\le \frac{\sum_{i=1}^m\|R_i^{(jj)}\|_{\F}^2}{\omega_{\robu}^2}
\le \frac{\sum_{i=1}^m\|\Bdiag_{\tau_n}(R_i)\|_{\F}^2}{\omega_{\robu}^2}.
\end{align*}
This completes the proof.
\end{proof}

Previously in Theorem~\ref{thm:main}, $Q$ is set to $W^{-1}\ww$, but the one in the next lemma can be any given
nonsingular matrix.

\begin{lemma}\label{lem2}
For any given nonsingular $Q\in{\mathbb R}^{n\times n}$,
let $Z=Q\Gamma Q^{-1}$ and write $Z= B - E$ with
$B=\Bdiag_{\tau_n}(Z)$ and $E=-\OffBdiag_{\tau_n}(Z)$.
Let $\tau$ and $\alpha$ be as in \eqref{eq:tau-alpha} and $g$ as in \eqref{eq:res-err}.
If
\begin{align}\label{econ}
g>\|Q^{-1}EQ\|_{\F}/\alpha,
\end{align}
then there exists a $\tau_n$-block diagonal matrix
$\wtd B=\diag(\wtd B_{11},\dots, \wtd B_{tt})$
and a nonsingular matrix $P=\big[P_{jk}\big]$ with $P_{jk}\in\mathbb{R}^{n_j\times n_k}$ and
 $P_{jj}=I_{n_j}$ such that
\begin{align}
B (QP) =(QP)\wtd B,
\end{align}
and for $j=1,2,\ldots,t$
\begin{subequations}\label{plam}
\begin{align}
&\|\widehat{P}_j\|_{\F}\le \frac {\tau}{\alpha}\cdot\frac {\|Q^{-1}EQ\|_{\F}}g,\label{plam1}\\  %\frac{2\|Q^{-1}EQ\|_{\F}}{g-\sqrt{2}\|Q^{-1}EQ\|_{\F}},\label{plam1}\\
&\sum_{k=1}^{n_j}|\tilde\mu_{jk}-\gamma_j|^2<%\frac t{t-1}\cdot
(1+\tau^2) \cdot \|Q^{-1}EQ\|_{\F}^2, \label{plam2} % \frac{g^2}{4},\label{plam2}
\end{align}
where $\tilde\mu_{j1},\dots, \tilde\mu_{jn_j}$ are the eigenvalues of $\wtd B_{jj}$, and
\begin{align}\label{hatpj}
\widehat{P}_j=\begin{bmatrix} P_{1j}^{\T}, & \dots, & P_{j-1,j}^{\T}, & 0_{n_j\times n_j}, &
                     P_{j+1,j}^{\T}, & \dots, & P_{tj}^{\T}
              \end{bmatrix}^{\T}.
\end{align}
\end{subequations}
\end{lemma}

\begin{proof}
It suffices to show there exist $\widehat{P}_1\in\mathbb{R}^{n\times n_1}$ and
$\wtd B_{11}\in\mathbb{R}^{n_1\times n_1}$
such that
\begin{align}\label{dqp1}
Q^{-1}BQ\begin{bmatrix} I_{n_1}\\ \widehat{P}_1 \end{bmatrix}
  \equiv (\Gamma+Q^{-1}EQ)\begin{bmatrix} I_{n_1}\\ \widehat{P}_1 \end{bmatrix} =\begin{bmatrix} I_{n_1}\\ \widehat{P}_1 \end{bmatrix} \wtd B_{11},
\end{align}
\eqref{plam} for $j=1$ holds,  and $P$ is nonsingular.

%First,  we have $Q^{-1}BQ = $.
Partition $Q^{-1}EQ=\begin{bmatrix} E_{11} & E_{12} \\ E_{21} & E_{22}\end{bmatrix}$ with $E_{11}\in\mathbb{R}^{n_1\times n_1}$, $E_{22}\in\mathbb{R}^{(n-n_1)\times (n-n_1)}$.
A direct calculation gives
\begin{align*}
\sep_{\F}(\gamma_1 I_{n_1}, \diag(\gamma_2 I_{n_2},\dots, \gamma_t I_{n_t}))
=\min_{2\le j\le t} |\gamma_j-\gamma_1|\ge g,
\end{align*}
where $\sep_{\F}(\cdots)$ is
the separation of two matrices \cite[p.247]{stewart1990matrix}. Let
$\tilde g=g-\|E_{11}\|_{\F}-\|E_{22}\|_{\F}$.
By \cite[Theorem 2.8 on p.238]{stewart1990matrix}, we conclude that
if
\begin{equation}\label{eq:lm2-pf-1}
\tilde g>0, \quad \frac {\|E_{21}\|_{\F}\|E_{12}\|_{\F}}{\tilde g^2}<\frac 14,
\end{equation}
then there is a unique $\what P_1\in{\mathbb R}^{(n-n_1)\times n_1}$ such that
\begin{equation}\label{eq:lm2-pf-2}
\|\what P_1\|_{\F}\le\frac {2\|E_{21}\|_{\F}}{\tilde g+\sqrt{\tilde g^2-4\|E_{21}\|_{\F}\|E_{12}\|_{\F}}}
   %<2\frac {\|E_{21}\|_{\F}}{\tilde g}
\end{equation}
and \eqref{dqp1} holds. We have to show that the assumption \eqref{econ} ensures \eqref{eq:lm2-pf-1} and that
\eqref{eq:lm2-pf-2} implies \eqref{plam1} for $j=1$. In fact, under \eqref{econ},
\begin{align}
\tilde g %&\ge g-(\|E_{11}\|_{\F}+\|E_{22}\|_{\F})  \nonumber\\
   &\ge g-\sqrt{2(\|E_{11}\|_{\F}^2 + \|E_{22}\|_{\F}^2)} \nonumber\\
   &\ge g-\sqrt{2}\|Q^{-1}EQ\|_{\F} \nonumber\\
   &>(1-\sqrt 2\,\alpha)g \label{eq:lm2-pf-3}\\
   &> 0, \nonumber \\
\frac{\|E_{21}\|_{\F}\|E_{12}\|_{\F}}{\tilde g^2}
    &\le\frac{\|E_{21}\|_{\F}^2+\|E_{12}\|_{\F}^2}{2\tilde{g}^2} \nonumber\\
    &<\frac{\|E_{21}\|_{\F}^2+\|E_{12}\|_{\F}^2}{2(1-\sqrt 2\,\alpha)^2g^2} \nonumber\\
    &\le\frac{\|Q^{-1}EQ\|_{\F}^2}{2(1-\sqrt 2\,\alpha)^2g^2} \nonumber\\
    &\le\frac {\alpha^2}{2(1-\sqrt 2\,\alpha)^2} \label{eq:lm2-pf-4}\\
    & <\frac 14. \nonumber
\end{align}
They give \eqref{eq:lm2-pf-1}. It follows from \eqref{eq:lm2-pf-2}, \eqref{eq:lm2-pf-3}, and \eqref{eq:lm2-pf-4} that
\begin{align}
\|\what P_1\|_{\F}&\le\frac {2}{(1-\sqrt 2\,\alpha)+\sqrt{(1-\sqrt 2\,\alpha)^2-2\alpha^2}}
                   \cdot\frac {\|Q^{-1}EQ\|_{\F}}g \nonumber \\
              &=\frac {\tau}{\alpha}\cdot\frac {\|Q^{-1}EQ\|_{\F}}g \label{eq:lm2-pf-5} \\
              &<\tau. \nonumber
\end{align}
The inequality \eqref{plam1} for $j=1$ is a result of \eqref{eq:lm2-pf-5}.

%Using \eqref{econ} and the definition of $\alpha$ in \eqref{ga}, we get
%\begin{align*}
%g-(\|E_{11}\|_{\F}+\|E_{22}\|_{\F})&,\\
%\frac{\sqrt{\|E_{21}\|_{\F}\|E_{12}\|_{\F}}}{g-(\|E_{12}\|_{\F} + \|E_{21}\|_{\F})}
%&\le\frac{\sqrt{\|E_{21}\|_{\F}^2+\|E_{12}\|_{\F}^2}}{\sqrt{2}(g-\sqrt{2}\|Q^{-1}EQ\|_{\F}^2)}
%< \frac{\|Q^{-1}EQ\|_{\F}}{2\sqrt{2(t-1)}\|Q^{-1}EQ\|_{\F}}<\frac12.
%\end{align*}
%By \cite[Theorem 4.11]{stewart1973error}, we know that there exists a
%$ \widehat{P}_1\in\mathbb{R}^{(n-n_1)\times n_1}$ satisfying \eqref{plam1}
%%with $\|\widehat{P}_1\|_{\F}\le \frac{2\|Q^{-1}EQ\|_{\F}}{g-\sqrt{2}\|Q^{-1}EQ\|_{\F}}$
%such that $\mathcal{R}(\begin{bmatrix} I_{n_1} \\ \widehat{P}_1\end{bmatrix})$ is an invariant subspace of
%$\Gamma+Q^{-1}EQ$, i.e., there exists a matrix $\wtd B_{11}\in\mathbb{R}^{n_1\times n_1}$ such that
%\begin{align}\label{gammaqeq}
%(\Gamma+Q^{-1}EQ)\begin{bmatrix} I_{n_1} \\ \widehat{P}_1\end{bmatrix}=\begin{bmatrix} I_{n_1} \\ \widehat{P}_1\end{bmatrix} \wtd B_{11}.
%\end{align}
%Then  it follows that
%\begin{align*}
%BQ\begin{bmatrix} I_{n_1} \\ \widehat{P}_1\end{bmatrix}
%=(Q\Gamma Q^{-1}+E)Q\begin{bmatrix} I_{n_1} \\ \widehat{P}_1\end{bmatrix}
%=Q(\Gamma +Q^{-1}EQ) \begin{bmatrix} I_{n_1} \\ \widehat{P}_1\end{bmatrix}
%=Q\begin{bmatrix} I_{n_1} \\ \widehat{P}_1\end{bmatrix}\wtd B_{11},
%\end{align*}
%which shows \eqref{dqp1}.

Next we show \eqref{plam2} for $j=1$.
%First, it is obvious that  $\widehat{P}_1$ satisfies \eqref{plam1}.
%Second,
Pre-multiply \eqref{dqp1} by $[I_{n_1}, 0]$ to get, after rearrangement,
\[
\wtd B_{11}-\gamma_1 I_{n_1}= [I_{n_1}, 0] Q^{-1} EQ \begin{bmatrix} I_{n_1} \\ P_1 \end{bmatrix}.
\]
Since $\lambda(\wtd B_{11})=\{\tilde\mu_{1k}\}_{k=1}^{n_1}$,
we have
\begin{align*}
\sum_{k=1}^{n_1}|\tilde\mu_{1k}-\gamma_1|^2
&\le \left\|[I_{n_1}\, 0] Q^{-1} EQ \begin{bmatrix} I_{n_1} \\ \widehat{P}_1 \end{bmatrix}\right\|_{\F}^2 \\
&\le \left\|\begin{bmatrix} I_{n_1} \\ \widehat{P}_1 \end{bmatrix}\right\|_2^2 \|Q^{-1}EQ\|_{\F}^2\\
&\le (1+\|\widehat{P}_1^{\T}\widehat{P}_1\|_2) \|Q^{-1}EQ\|_{\F}^2 \\
&\le (1+\tau^2)\cdot \|Q^{-1}EQ\|_{\F}^2, %\\
%&= \frac t{t-1}\cdot\|Q^{-1}EQ\|_{\F}^2,
\end{align*}
as was to be shown.

Finally, we show that $P$ is nonsingular by contradiction.
If $P$ were singular, let $x=[x_1^{\T}\, \dots \, x_t^{\T}]^{\T}$ be a nonzero vector
with $x_j\in\mathbb{R}^{n_j}$ such that $Px=0$.
%Partition $P$ as $P_{jk}$ with $P_{jk}\in\mathbb{R}^{n_j\times n_k}$.
We then have $x_j=-\sum_{\substack{k=1 \\ k\ne j}}^t P_{jk}x_k$ and thus
\[
\|x_j\|_2^2=\Big(\Big\|\sum_{\substack{k=1 \\ k\ne j}}^t P_{jk}x_k\Big\|_2\Big)^2
   \le \Big(\sum_{\substack{k=1 \\ k\ne j}}^t \|P_{jk}\|_2 \|x_k\|_2\Big)^2
   \le  (t-1) \sum_{\substack{k=1 \\ k\ne j}}^t \|P_{jk}\|_2^2 \|x_k\|_2^2.
\]
Therefore
\begin{align*}
\|x\|_2^2=\sum_{j=1}^t\|x_j\|_2^2
   &\le (t-1)\sum_{j=1}^t  \sum_{\substack{k=1 \\ k\ne j}}^t \|P_{jk}\|_2^2 \|x_k\|_2^2\\
   &= (t-1)\sum_{k=1}^t  \sum_{\substack{j=1 \\ j\ne k}}^t  \|P_{jk}\|_2^2 \|x_k\|_2^2 \\
   &\le (t-1)\sum_{k=1}^t    \|\widehat{P}_k\|_{\F}^2 \|x_k\|_2^2\\
   &< (t-1) \tau^2 \|x\|_2^2 < \|x\|_2^2,
\end{align*}
a contradiction. This completes the proof.
\end{proof}

\begin{remark}
Lemma~\ref{lem2} implies that when the off-block diagonal part of $Z$ is sufficiently small,
$QP$ is the eigenvector matrix of $B=\Bdiag_{\tau_n}(Z)$ with $P\approx I$,
and for each $j$ there are $n_j$ eigenvalues of $B$ that cluster around $\gamma_j$.
\end{remark}

\begin{lemma}\label{lem3}
Let $P=\big[P_{jk}\big]$ with $P_{jk}\in\mathbb{R}^{n_j\times n_k}$,
$P_{jj}=I_{n_j}$, and $\|\widehat{P}_j\|_{\F}\le \epsilon$,
where $\what P_j$ is defined as in \eqref{hatpj}, $0\le \epsilon<\tau$, and $\tau$ is defined by \eqref{eq:tau-alpha}.
Then
\begin{align}\label{pmi}
\|P-I\|_{\F}\le \sqrt{t}\,\epsilon.
\end{align}
Furthermore, let $W$, $\ww\in\mathbb{W}_{\tau_n}$,
$\wtd D=\diag(\wtd D_{11},\dots, \wtd D_{tt})\in\mathbb{D}_{\tau_n}$, and $\Pi\in{\mathbb P}_{\tau_n}$.
If $W\wtd D=\ww P\Pi$,
then $\wtd D$ is nonsingular and
\begin{align}\label{dsig}
\sqrt{1-2 \sqrt{t-1}\,\epsilon - (t-1)\epsilon^2} \le \sigma\le\sqrt{1+2 \sqrt{t-1}\,\epsilon + (t-1)\epsilon^2}.
\end{align}
for each singular value $\sigma$ of $\wtd D$.

\end{lemma}

\begin{proof}
Since $P-I=\left[ \widehat{P}_1, \dots, \widehat{P}_t \right]$, we have
$$
\|P-I\|_{\F}
%=\left\|\begin{bmatrix} \widehat{P}_1 & \dots \widehat{P}_t \end{bmatrix}\right\|_2
%  =\Big\|\sum_{j=1}^t \widehat{P}_j\widehat{P}_j^{\T}\Big\|_2^{1/2}
% \le \left(\sum_{j=1}^t \Big\|\widehat{P}_j\widehat{P}_j^{\T}\Big\|_2\right)^{1/2}
% \le
 =\left(\sum_{j=1}^t \Big\|\widehat{P}_j\Big\|_{\F}^2\right)^{1/2}
  \le \sqrt{t} \epsilon,
$$
which is \eqref{pmi}.

Next we show that $\wtd D$ is nonsingular and \eqref{dsig} holds.
Write $P=\begin{bmatrix} P_1, & \dots, & P_t \end{bmatrix}$ with $P_j\in\mathbb{R}^{n\times n_j}$.
Using $W\wtd D=\ww P\Pi$, we get
\begin{equation}\label{eq:lm3-pf-1}
\wtd D^{\T}W^{\T}W\wtd D=\Pi^{\T}P^{\T}\ww^{\T} \ww P\Pi.
\end{equation}
Since $W\in\mathbb{W}_{\tau_n}$,
the $j$th diagonal blocks at both sides of \eqref{eq:lm3-pf-1} read
\begin{align}\label{ddpwwp}
\wtd D_{jj}^{\T} \wtd D_{jj} =P_{j'}^{\T} \ww^{\T} \ww P_{j'},
\end{align}
where $1\le j'\le t$ as a result of the permutation $\Pi$.
Partition $\ww$ as $\ww=\big[\ww_1, \dots , \ww_t\big]$ with
$\ww_j\in\mathbb{R}^{n\times n_j}$.
We infer from $\ww\in\mathbb{W}_{\tau_n}$ that $\ww_j^{\T}\ww_j=I_{n_j}$
and $\big\|\ww_j^{\T}\ww_{\ell}\big\|_2\le 1$. To see the last inequality, we note
\begin{equation}\label{eq:WjWl}
|x_j^{\T}\ww_j^{\T}\ww_{\ell}x_{\ell}|
  \le\|\ww_jx_j\|_2\|\ww_{\ell}x_{\ell}\|_2=\|x_j\|_2\|x_{\ell}\|_2=1
\end{equation}
for any unit vectors $x_j\in\mathbb{R}^{n_j}$ and $x_{\ell}\in\mathbb{R}^{n_{\ell}}$.
Now using $P_{j'j'}=I_{n_{j'}}$ and $\|\widehat{P}_{j'}\|_{\F}\le \epsilon$,
we have
\begin{align*}
\big\|P_{j'}^{\T} \ww^{\T} \ww P_{j'}-I_{n_{j'}}\big\|_{\F}
%&=\Big\|\sum_{k=1}^{t}\sum_{\ell=1}^t P_{kj'}^{\T} \ww_k^{\T} \ww_\ell P_{\ell j'} - I_{n_{j'}}\Big\|_{\F} \\
&=\big\| \ww_{j'}^{\T} \ww \widehat{P}_{j'}
+  \widehat{P}_{j'}^{\T} \ww^{\T} \ww_{j'}
+\widehat{P}_{j'}^{\T} \ww^{\T} \ww \widehat{P}_{j'} \big\|_{\F}\\
&\le 2  \Big\|\sum_{\ell \ne j'}  \ww_{j'}^{\T} \ww_{\ell} P_{\ell j'}\Big\|_{\F} +
\Big\|\sum_{k\ne j'}\sum_{\ell\ne j'} P_{kj'}^{\T} \ww_k^{\T} \ww_\ell P_{\ell j'}\Big\|_{\F} \\
&\le 2  \sum_{\ell \ne j'}  \| P_{\ell j'}\|_{\F}+
\sum_{k\ne j'}\sum_{\ell \ne j'} \big\|P_{kj'}\big\|_{\F} \big\|P_{\ell j'}\big\|_{\F}\\
&=2  \sum_{\ell \ne j'}  \| P_{\ell j'}\|_{\F}+ \Big(\sum_{k\ne j'}\big\|P_{kj'}\big\|_{\F}\Big)^2\\
&\le  2\Big[(t-1)\sum_{k\ne j'}\big\|P_{kj'}\big\|_{\F}^2\Big]^{1/2}
        +(t-1)\sum_{k\ne j'}\big\|P_{kj'}\big\|_{\F}^2\\
&\le 2 \sqrt{t-1}\, \epsilon + (t-1)\epsilon^2.
\end{align*}
Combining it with \eqref{ddpwwp}, we get
\[
\|\wtd D_{jj}^{\T}\wtd D_{jj}-I_{n_j}\|_{\F}\le 2 \sqrt{t-1}\, \epsilon + (t-1)\epsilon^2
<2 \sqrt{t-1} \tau + (t-1)\tau^2=1,
\]
which implies that $\wtd D_{jj}$ is nonsingular,
and for any singular value $\sigma$ of $\wtd D_{jj}$, it holds that
\[
-1<- 2 \sqrt{t-1}\,\epsilon - (t-1)\epsilon^2 \le \sigma^2 -1\le 2 \sqrt{t-1}\,\epsilon + (t-1)\epsilon^2<1.
\]
The  conclusion follows immediately since $\wtd D\in\mathbb{D}_{\tau_n}$.
\end{proof}

We now present a proof of \eqref{eq:W-norm}. Since $\|\ww\|_2$ is equal to
the square root of the largest eigenvalue of $\ww^{\T}\ww$ and the latter is no smaller than the largest
diagonal entry of $\ww^{\T}\ww$, we have $\|\ww\|_2\ge 1$. Let
$x=\big[x_1^{\T},x_2^{\T},\ldots,x_t^{\T}\big]^{\T}$ with
$x_j\in\mathbb{R}^{n_j}$. Similarly to \eqref{eq:WjWl}, we find
$$
x^{\T}\ww^{\T}\ww x
   =\sum_{j,\,\ell}x_j^{\T}\ww_j^{\T}\ww_{\ell}x_{\ell}
   \le\sum_{j,\,\ell}\|x_j\|_2\|x_{\ell}\|_2
   \le\frac 12\sum_{j,\,\ell}\big(\|x_j\|_2^2+\|x_{\ell}\|_2^2\big)=t\|x\|_2^2,
$$
and thus $\|\ww\|_2\le\sqrt t$.

\subsection{Proof of Theorem~\ref{thm:main}}\label{subsec:proof}
Recall $Q=W^{-1}\ww$ and let $Z=Q\Gamma Q^{-1}$. Partition $Z=\big[Z_{jk}\big]$ with $Z_{jk}\in\mathbb{R}^{n_j\times n_k}$,
       and let $\lambda(Z_{jj})=\{\mu_{jk}\}_{k=1}^{n_j}$.
The proof will be completed in the following four steps:
\begin{itemize}
\item[Step 1.] We will show that $Z$ is approximately  $\tau_n$-block diagonal. Specifically, we show
       \begin{align}\label{offbz}
       \|\OffBdiag_{\tau_n}(Z)\|_{\F}
           \le\frac{\left(\sum_{i=1}^m\|\OffBdiag_{\tau_n}(R_i)\|_{\F}^2\right)^{1/2}}{\omega_{\uniq}}
           \le\frac{\delta}{\omega_{\uniq}},
       \end{align}
       where $R_i$ is given by \eqref{eq:Ri}.

\item[Step 2.]
       We will show that the eigenvalues of $Z_{jj}$ cluster around a unique $\gamma_{j'}$
       by showing that there exists a permutation $\pi$ of $\{1,2,\dots, t\}$
       such that
       \begin{align}\label{mujk}
       |\mu_{jk}-\gamma_{\pi(j)}|<\frac{g}{2},\quad
       |\mu_{jk}-\gamma_i|>\frac{g}{2},\quad  \mbox{for any $i\ne \pi(j)$}.
       \end{align}
       In the other word, each of the $t$ disjoint intervals $(\gamma_i-g/2,\gamma_i+g/2)$ contains
       one and only one $\lambda(Z_{jj})$.

\item[Step 3.] We will show that there exist
       a  permutation $\Pi\in{\mathbb P}_{\tau_n}$ and
       a nonsingular $P\equiv \big[P_{jk}\big]\in{\mathbb R}^{n\times n}$  with $P_{jk}\in\mathbb{R}^{n_j\times n_k}$ and
        $P_{jj}=I_{n_j}$, satisfying \eqref{plam1},
       such that
       $\wtd D=QP\Pi\in{\mathbb D}_{\tau_n}$.
       %Note here that $\wtd D$ and $Q$ are respectively the eigenvector matrices of $\Gamma$ and $Z$,
       %which are the exact and approximate solutions to \eqref{azza}, respectively.
       %Therefore, $\wtd D=QP\Pi$ essentially establishes the relationship
       %between the eigenvector matrices of the exact and approximate solutions to \eqref{azza}.

\item[Step 4.] We will prove \eqref{ineq:main}.
\end{itemize}
%Let $Z=Q\Gamma Q^{-1}$.

%What follows we prove the results in the above four steps one by one.

\smallskip
\noindent{\em Proof\/} of Step 1.
Recall $\wtd R_i=\ww^{\T}\wtd A_i \ww \Gamma- \Gamma \ww^{\T} \wtd A_i \ww$ of \eqref{eq:diag-res}. We have
\begin{align*}
\wtd R_i%&=\ww^{\T}\wtd A_i \ww \Gamma- \Gamma \ww^{\T} \wtd A_i \ww\\
&=\ww^{\T}A_i \ww \Gamma- \Gamma \ww^{\T} A_i \ww
+\ww^{\T}\Delta A_i \ww \Gamma- \Gamma \ww^{\T} \Delta A_i \ww\\
&=Q^{\T} W^{\T}A_iW Q\Gamma -\Gamma Q^{\T} W^{\T} A_iW Q
+\ww^{\T}\Delta A_i \ww \Gamma- \Gamma \ww^{\T} \Delta A_i \ww,
\end{align*}
from which it follows that
\begin{align*}
 R_i&=W^{\T} A_i W Z - Z^{\T} W^{\T} A_i W\\
     &= Q^{-\T}\wtd R_iQ^{-1}-W^{\T} \Delta A_i \ww\Gamma Q^{-1} + Q^{-\T}\Gamma \ww^{\T} \Delta A_i W.
\end{align*}
Putting all of them for $1\le i\le m$ together, we get
%\begin{align*}
%\|R_i\|_{\F}\le
%\|T^{-1}\|_2^2\|\wtd R_i\|_{\F}+2 \|T^{-1}\|_2\|W\|_2\|\ww\|_2 \|\Delta A_i\|_{\F}.
%\end{align*}
\begin{align*}
\begin{bmatrix} R_1\\ \vdots \\ R_m\end{bmatrix}
=(I_m\otimes Q^{-\T} )  \begin{bmatrix} \wtd R_1\\ \vdots \\ \wtd R_m\end{bmatrix} Q^{-1}
& - (I_m\otimes W^{\T}) \begin{bmatrix} \Delta A_1\\ \vdots \\ \Delta A_m\end{bmatrix} \ww^{\T}\Gamma Q^{-1}\\
&+ \big[I_m\otimes (Q^{-\T}\Gamma\ww^{\T})\big] \begin{bmatrix} \Delta A_1\\ \vdots \\ \Delta A_m\end{bmatrix} W.
\end{align*}
%By simple calculations, we get
Consequently,
\begin{align*}
\left(\sum_{i=1}^m\|R_i\|_{\F}^2\right)^{1/2}
\le \|Q^{-1}\|_2^2\, \tilde r
+2 \|Q^{-1}\|_2\|W\|_2\|\ww\|_2 \delta_{\cA}= \delta.
\end{align*}
Combine it with \eqref{offz} in Lemma~\ref{lem1} to conclude \eqref{offbz}.
\qed

\smallskip
\noindent{\em Proof\/} of Step 2.
Using Lemma~\ref{lem1}, we know that there exists $\hat\mu_j$ such that
\begin{align}\label{summu}
\sum_{k=1}^{n_j}|\mu_{jk} -\hat\mu_j|^2
\le\frac{\sum_{i=1}^m\|\Bdiag_{\tau_n}(R_i)\|_{\F}^2}{\omega_{\robu}^2}
\le \left(\frac{\delta}{\omega_{\robu}}\right)^2.
\end{align}
Then for any $\mu_{j\,k_1}$, $\mu_{j\,k_2}$,
we have
\begin{align}
|\mu_{j\,k_1}-\mu_{j\,k_2}|^2
&\le (|\mu_{j\,k_1} - \hat\mu_j| + |\mu_{j\,k_2} - \hat\mu_j|)^2 \label{mumu}\\
&\le 2(|\mu_{j\,k_1} - \hat\mu_j|^2 + |\mu_{j\,k_2} - \hat\mu_j|^2)\notag\\
&\le 2\sum_{k=1}^{n_j} | \mu_{jk} -\hat\mu_j|^2
\le 2 \left(\frac{\delta}{\omega_{\robu}}\right)^2. \nonumber
\end{align}
Let $\argmin_{\ell}|\mu_{jk}-\gamma_{\ell}|=\ell_{jk}$.
Noticing that
\[
\Gamma= Q^{-1}ZQ=Q^{-1} \Bdiag_{\tau_n}(Z)Q +Q^{-1}\OffBdiag_{\tau_n}(Z)Q.
\]
By a result of Kahan~\cite{kahan1975spectra} (see also  \cite[Remark~3.3]{sun1996variation}),
we have
\begin{align}\label{eigerror}
\sum_{j=1}^t\sum_{k=1}^{n_j}|\mu_{jk} - \gamma_{\ell_{jk}}|^2\le 2\|Q^{-1}\OffBdiag_{\tau_n}(Z)Q\|_{\F}^2.
\end{align}
Now we declare $\ell_{j1}=\dots=\ell_{jn_j}=j'$ for all $j=1,2,\ldots,t$.
Because otherwise, say $\ell_{j1}\ne \ell_{j2}$,
we have
\begin{subequations}
\begin{alignat}{2}
4\alpha^2 g^2
&>4\kappa^2_2(Q)\frac{\delta^2}{\omega_{\uniq}^2} &&\quad\mbox{(by \eqref{eq:pert-tiny-cond})} \nonumber\\
&\ge 4\|Q^{-1}\OffBdiag_{\tau_n}(Z)Q\|_{\F}^2 &&\quad\mbox{(by \eqref{offbz})} \label{gg1}\\
&\ge 2\sum_{j=1}^t\sum_{k=1}^{n_j}|\mu_{jk} - \gamma_{\ell_{jk}}|^2 &&\quad\mbox{(by \eqref{eigerror})} \nonumber\\
&\ge 2(|\mu_{j1} - \gamma_{\ell_{j1}}|^2 + |\mu_{j2} - \gamma_{\ell_{j2}}|^2) &&\nonumber\\
&\ge (|\mu_{j1} - \gamma_{\ell_{j1}}| + |\mu_{j2} - \gamma_{\ell_{j2}}|)^2 && \nonumber\\
&\ge(  | \gamma_{\ell_{j1}} - \gamma_{\ell_{j2}}| - |\mu_{j1}  - \mu_{j2}| )^2 && \notag\\
&\ge \left(g-\sqrt{2}\,\frac{\delta}{\omega_{\robu}}\right)^2 &&\quad\mbox{(by \eqref{mumu})} \nonumber\\
&>[1-(1-2\alpha)]^2 g^2 &&\quad\mbox{(by \eqref{eq:pert-tiny-cond})} \nonumber\\
&= 4\alpha^2 g^2,\label{gg3}
\end{alignat}
\end{subequations}
a contradiction.
%Here the  first and second inequalities
%of \eqref{gg1} use \eqref{eq:pert-tiny-cond} and \eqref{offbz}, respectively,
%the first inequality of\eqref{gg2} uses \eqref{eigerror},
%the first and second inequalities of \eqref{gg3} use
%$|\gamma_{\ell_{j1}} -\gamma_{\ell_{j2}}|\ge g$ and \eqref{eq:pert-tiny-cond}, respectively.
%Using similar argument, we can also declare $n_j=n_{j'}$.
Now using \eqref{eigerror}, \eqref{offbz} and \eqref{eq:pert-tiny-cond}, we get
\begin{align*}
\max_{k}|\mu_{jk}-\gamma_{j'}|
&\le\left(\sum_{k=1}^{n_j}|\mu_{jk}-\gamma_{j'}|^2\right)^{1/2}
\le \sqrt{2}\|Q^{-1}\OffBdiag_{\tau_n}(Z)Q\|_{\F}\\
&\le\sqrt{2}\kappa_2(Q)\|\OffBdiag_{\tau_n}(Z)\|_{\F}
\le \frac{\sqrt{2}\kappa_2(Q)\delta}{\omega_{\uniq}}
<\sqrt{2}\alpha g<\frac12 g.
\end{align*}
Thus, we know that each $j\in\{1,2,\ldots, t\}$ corresponds to a unique ${j'}$
satisfying that
$|\mu_{jk}-\gamma_{j'}|<{g}/{2}$ and
$|\mu_{jk}-\gamma_i|>{g}/{2}$ for any $i\ne j'$.
This is \eqref{mujk}.
\qed

\smallskip
\noindent{\em Proof\/} of Step 3.
Notice that \eqref{gg1} implies that $\|Q^{-1}\OffBdiag_{\tau_n}(Z)Q\|_{\F}\le \alpha g$,
i.e.,  \eqref{econ} holds.
By Lemma~\ref{lem2},
there exists a $\tau_n$-block diagonal matrix
$\wtd B=\diag(\wtd B_{11},\dots, \wtd B_{tt})$
and a nonsingular matrix $P\equiv\big[P_{jk}\big]$ with $P_{jk}\in\mathbb{R}^{n_j\times n_k}$ and
 $P_{jj}=I_{n_j}$, satisfying \eqref{plam}, such that
\begin{equation}\label{eq:main-pf-1}
\Bdiag_{\tau_n}(Z) (QP) =(QP)\wtd B.
\end{equation}
Denote by $\lambda(\wtd B_{jj})=\{\tilde\mu_{jk}\}_{k=1}^{n_j}$.
By \eqref{plam2}, \eqref{offbz} and \eqref{eq:pert-tiny-cond}, we know
\begin{align*}
\max_k|\tilde\mu_{jk}-\gamma_j|
&\le\sqrt{\sum_k|\tilde\mu_{jk}-\gamma_j|^2} \\
&\le (1+\tau^2)\kappa_2(Q)\|\OffBdiag_{\tau_n}(Z)\|_{\F}\\
&<(1+\tau^2) \kappa_2(Q) \frac{\delta}{\omega_{\uniq}}< (1+\tau^2)\alpha g<\frac{g}{2}.
\end{align*}
What this means is that each of the $t$ disjoint intervals $(\gamma_i-g/2,\gamma_i+g/2)$ contains
one and only one $\lambda(\wtd B_{jj})$. Previously in Step 2,
we proved that each of the $t$ disjoint intervals $(\gamma_i-g/2,\gamma_i+g/2)$ contains
one and only one $\lambda(Z_{jj})$ as well. On the other hand, we also have
$\lambda(\Bdiag_{\tau_n}(Z))=\lambda(\wtd B)$ by \eqref{eq:main-pf-1}. Therefore, there is permutation
$\pi$ of $\{1,2,\ldots,t\}$ such that
\begin{equation}\label{eq:main-pf-2}
\lambda(\wtd B_{\pi(j)\pi(j)})=\lambda(Z_{jj})\quad\mbox{for $1\le j\le t$}.
\end{equation}
Let $\Pi$ be the permutation
matrix such that
\begin{equation}\label{eq:main-pf-2a}
\Pi^{\T} \wtd B \Pi=\diag(\wtd B_{\pi(1)\pi(1)},\dots,\wtd B_{\pi(t)\pi(t)}).
\end{equation}
It can be seen that $\Pi\in{\mathbb P}_{\tau_n}$, i.e., it is
$\tau_n$-block structure preserving. Finally by \eqref{eq:main-pf-2} and \eqref{eq:main-pf-2a},
\begin{align}
\diag(Z_{11},\dots,Z_{tt})(QP \Pi)
   &=QP\wtd B \Pi \label{eq:main-pf-3}\\
   &=(QP\Pi) \Pi^{\T} \wtd B \Pi \nonumber \\
   &=(QP\Pi)\diag(\wtd B_{\pi(1)\pi(1)},\dots,\wtd B_{\pi(t)\pi(t)}). \nonumber
\end{align}
Let $\wtd D=QP\Pi\equiv\big[\wtd D_{jk}\big]$ with $\wtd D_{jk}\in\mathbb{R}^{n_j\times n_k}$. The equation \eqref{eq:main-pf-3} becomes
\[
\diag(Z_{11},\dots,Z_{tt})\wtd D
=\wtd D\diag(\wtd B_{\pi(1)\pi(1)},\dots,\wtd B_{\pi(t)\pi(t)})
\]
which yields $Z_{jj}\wtd D_{jk}=\wtd D_{jk}\wtd B_{\pi(k)\pi(k)}$.
Recalling \eqref{eq:main-pf-2} and
$\lambda(Z_{jj})\cap\lambda(Z_{kk})=\emptyset$ for $j\ne k$ by \eqref{mujk}, we conclude
that $\wtd D_{jk}=0$ for $j\ne k$,
i.e., $\wtd D$ is $\tau_n$-block diagonal.
\qed

\smallskip
\noindent{\em Proof\/} of Step 4.
Noticing that $Q=W^{-1}\ww$ and $\wtd D=QP\Pi$ in Step 3,
we have $W\wtd D=\ww P\Pi$.
Then using Lemma~\ref{lem3},
we know that $\wtd D$ is nonsingular and for any singular value $\sigma$ of $\wtd D$, and \eqref{dsig} holds  with
\[
\epsilon= \frac{\tau}{\alpha}\cdot\frac{\|Q^{-1}\OffBdiag_{\tau_n}(Z)Q\|_{\F}}{g}.
\]
By \eqref{offbz}, we have
\begin{equation}\label{epsub}
\epsilon \le \frac{\tau}{\alpha}\cdot \frac{\kappa_2(Q)\delta}{g\,\omega_{\uniq}}=\epsilon_*.
\end{equation}
Now let  $\wtd D_{jj}=U_j\Sigma_j V_j^{\T}$ be the SVD of $\wtd D_{jj}$.
Denote by
$U=\diag(U_1,\dots, U_t)$, $V=\diag(V_1,\dots, V_t)$ and $D=\Pi VU^{\T} \Pi^{\T}$.
It can be verified that $D$ is orthogonal and $\tau_n$-block diagonal.
It follows from $W\wtd D=\ww P\Pi$ that
\begin{align*}
W=\ww P \Pi \wtd D^{-1}
&=\ww (\Pi \wtd D^{-1}\Pi^{\T} )\Pi + \ww \OffBdiag_{\tau_n}(P)\Pi \wtd D^{-1}\\
&=\ww D\Pi + \ww (\Pi \wtd D^{-1}\Pi^{\T} -D)\Pi + \ww \OffBdiag_{\tau_n}(P)\Pi \wtd D^{-1}\\
&=\ww D\Pi + \ww \Pi V( \Sigma^{-1} - I)U + \ww \OffBdiag_{\tau_n}(P)\Pi \wtd D^{-1}.
\end{align*}
Using Lemma~\ref{lem3}, we have for ${\scriptstyle\p}\in\{2,{\scriptstyle\F}\}$
\begin{align*}
\big\|W - \ww D\Pi \big\|_{\p}
&=\big\| \ww \Pi V( \Sigma^{-1} - I)U + \ww \OffBdiag_{\tau_n}(P)\Pi \wtd D^{-1} \big\|_{\p}\\
&\le \big\|\ww\big\|_{\p} \left( \frac{1+\sqrt{t}\,\epsilon_*}{\sqrt{1-2\sqrt{t-1}\epsilon_* -
(t-1)\epsilon_*^2}} - 1\right)\\
   &= \big\|\ww\big\|_{\p}\big[(\sqrt{t}+\sqrt{t-1})\epsilon +O(\epsilon^2)\big].
\end{align*}
Combine it with \eqref{epsub} to conclude the proof of \eqref{ineq:main}.
\qed

\section{Numerical examples}\label{sec:numer}
In  this section, we present some random numerical tests to validate our theoretical results.
All numerical examples were carried out using {\sc matlab}, with machine unit roundoff $2^{-53}\approx 1.1\times 10^{-16}$.

%\begin{example}
Let us start by explain how the testing examples are constructed.
Given a partition $\tau_n=(n_1,\dots,n_t)$ of $n$ and the number $m$ of matrices, we generate
the matrix sets $\cA=\{A_i\}_{i=1}^m$ and $\wtd\cA=\{\wtd A_i\}_{i=1}^m$
as follows.
\begin{enumerate}
\item Randomly generate  $W\equiv[W_1, \dots, W_t]\in\mathbb{W}_{\tau_n}$.
      This is done by first generating an  $n\times n$ random matrix from the standard normal distribution
      and then orthonormalizing its first $n_1$ columns, the next $n_2$ columns, $\ldots$, and the last $n_t$
      columns, respectively.
      Set $V=W^{-\T}$;

\item Generate $m$ $\tau_n$-block diagonal matrices $D_j$ randomly
      from the standard normal distribution and set $A_j=V D_j V^{\T}$ for $1\le j\le m$. This makes sure that
      $\cA$ is $\tau_n$-block diagonalizable.

\item Generate $m$ noise matrices $N_j$ also randomly
      from the standard normal distribution and set $\wtd A_j=A_j + \xi N_j$, where
      $\xi$ is a parameter for controlling noise level. $\wtd\cA$ is likely not $\tau_n$-block diagonalizable
      but it is approximately. An approximate block diagonalizer $\ww\equiv[\wtd W_1, \dots,\wtd W_t]\in\mathbb{W}_{\tau_n}$
      of $\wtd\cA$ is computed by JBD-NCG \cite{nion2011tensor} followed
      by orthonormalization as in item (1) above.
\end{enumerate}
For comparison purpose, we estimate the relative error between $\ww$ and $W$ as measured by \eqref{eq:measure}
for ${\scriptstyle\p}={\scriptstyle\F}$ as follows. We have to minimize
$$
\|W-\ww D\Pi\|_{\F}^2=\|W\|_{\F}^2-2\trace(W^{\T}\ww D\Pi)+\|\ww\|_{\F}^2
$$
over orthogonal $D\in{\mathbb D}_{\tau_n}$ and $\Pi\in{\mathbb P}_{\tau_n}$, which is equivalent to
maximizing
$$
\sum_{j=1}^t\trace(W_j^{\T}\ww_{\pi(j)}D_{\pi(j)}\Pi_j)
$$
over orthogonal $D_{\pi(j)}$, permutations $\pi$ of $\{1,2,\ldots,t\}$, subject to $n_j=n_{\pi(j)}$,
which again is equivalent to
\begin{equation}\label{eq:relerr-W-comp}
\max_{\pi}\sum_{j=1}^t(\mbox{the sum of the singular values of $W_j^{\T}\ww_{\pi(j)}$})
\end{equation}
subject to $n_j=n_{\pi(j)}$. Abusing notation a little bit, we let $\pi$ be the one that achieve the optimal in
\eqref{eq:relerr-W-comp}, perform the singular value decomposition  $\ww_{\pi(j)}^{\T}W_j=U_j\Sigma_j V_j^{\T}$,
and set $D=\diag(U_{\pi(1)}V_{\pi(1)}^{\T},\dots, U_{\pi(t)}V_{\pi(t)}^{\T})$.
Finally, the error \eqref{eq:measure} for ${\scriptstyle\p}={\scriptstyle\F}$ is given by
\begin{equation}\label{eq:measure-est}
\frac{\|W-\ww D\Pi\|_{\F}}{\|\ww\|_{\F}}
\end{equation}
with $D$ as above and $\Pi\in\mathbb{P}_{\tau_n}$ as determined by the optimal $\pi$. There doesn't seem to be a simple way to compute
\eqref{eq:measure} for ${\scriptstyle\p}={\scriptstyle 2}$.

To generate error bounds by Theorem~\ref{thm:main}, we have to decide what $\Gamma$ to use. Ideally, we should use
the one that minimize the right-hand side of \eqref{ineq:main}, but we don't have an simple way to do that.
For the tests below, we use 50 different $\Gamma$ and pick the best bound. Specifically,
we use a particular one
\begin{equation}\label{eq:Gamma0}
\Gamma=\diag(-1, -1+ \frac{2}{t-1}, -1+ \frac{4}{t-1},\dots, 1)
\end{equation}
as well as $49$ random ones with their diagonal entries $\gamma_1,\dots,\gamma_t$ randomly drawn from the interval $(-1,1)$
with the uniform distribution. Our experience suggests that the particular
$\Gamma$ in \eqref{eq:Gamma0} usually leads to bounds having the same order as the best one produced by the $49$ random $\Gamma$.
However, it can happen that the best one is much better than and up to one tenth of
than by the particular $\Gamma$,  although such  extremes do not happen very often.

We will report our numerical tests according to five different testing scenarios:
varying numbers of matrices (test 1),
varying matrix sizes (test 2),
varying numbers of diagonal blocks (test 3),
varying noise levels (test 4),
and varying condition numbers $\cond(\cA)$ (test 5).
We will examine these quantities:
the modulus of uniqueness $\omega_{\uniq}$, the modulus of non-divisibility $\omega_{\robu}$,
$\delta$ as defined in \eqref{eq:del},
the {\em ratio\/} as the quotient of $\delta$ over the right hand side of \eqref{eq:pert-tiny-cond}
(to make sure that \eqref{eq:pert-tiny-cond} is satisfied),
$\varepsilon_{\berr}\equiv\varepsilon_{\berr}(\wtd\cA;\wtd W)$ the upper bound as in \eqref{eq:backerr2} for the backward error,
$\cond({\cA})$ the condition number as defined in \eqref{condnum},
$\varepsilon_{\ub}$ as  in \eqref{ineq:main}, and finally
the {\em error\/} in $\wtd W$ as in \eqref{eq:measure-est}.

\begin{table}[ht]
\belowrulesep=1pt
\aboverulesep=1pt
  \centering
{\small
\begin{tabular}{c||c|c|c|c|c|c|c|c}\toprule
 $m$ & $\omega_{\uniq}$ & $\omega_{\robu}$  & $\delta$ & {\em ratio\/} & $\varepsilon_{\berr}$ &  $\cond(\cA)$ & $\varepsilon_{\ub}$ & {\em error} \\ \hline\hline
 4 & 1.7e+00 & 1.9e+00 & 4.8e-10 & 1.4e-09 & 3.4e-10 & 2.4e+03 & 1.3e-09 & 1.9e-11 \\ \hline
 8 & 3.8e+00 & 3.9e+00 & 2.2e-10 & 1.5e-09 & 3.2e-10 & 1.6e+03 & 1.4e-09 & 1.9e-11 \\ \hline
 16 & 6.6e+00 & 6.4e+00 & 9.8e-10 & 7.3e-10 & 3.3e-10 & 1.3e+03 & 6.8e-10 & 1.9e-11 \\ \hline
 32 & 1.0e+01 & 1.0e+01 & 8.5e-10 & 6.5e-10 & 2.7e-10 & 1.2e+03 & 6.0e-10 & 1.8e-11 \\ \hline
 64 & 1.6e+01 & 1.6e+01 & 1.3e-09 & 4.2e-10 & 1.8e-10 & 1.2e+03 & 3.8e-10 & 1.2e-11 \\ \hline
 128 & 2.5e+01 & 2.5e+01 & 2.2e-09 & 4.4e-10 & 2.1e-10 & 1.2e+03 & 4.0e-10 & 1.4e-11 \\ \hline
 256 & 3.6e+01 & 3.6e+01 & 1.8e-09 & 4.2e-10 & 1.7e-10 & 1.2e+03 & 3.9e-10 & 1.1e-11
  \\ \bottomrule
\end{tabular}
}
\caption{\small Bound vs. $m$, the number of matrices in $\cA$ for  $\tau_9=(3,3,3)$}
\label{tab1}
\end{table}

\smallskip
\noindent{\bf Test 1: number of matrices.}\quad
In this test, we fix $\xi=10^{-12}$ and vary the number $m$ of  matrices in the matrix set $\cA$.
The numerical results are displayed in Tables~\ref{tab1} and \ref{tab2} for the two different partitions $\tau_9=(3,3,3)$ and $\tau_6=(1,2,3)$,
respectively.
\begin{table}[ht]
  \centering
{\small
\begin{tabular}{c||c|c|c|c|c|c|c|c}\toprule
 $m$ & $\omega_{\uniq}$ & $\omega_{\robu}$ & $\delta$ & {\em ratio\/} & $\varepsilon_{\berr}$ &  $\cond(\cA)$ & $\varepsilon_{\ub}$ & {\em error} \\ \hline\hline
 4 & 8.1e-01 & 2.7e+00 & 1.4e-10 & 8.8e-10 & 3.6e-11 & 9.7e+04 & 8.1e-10 & 7.3e-12 \\ \hline
 8 & 3.0e+00 & 4.7e+00 & 1.7e-10 & 5.6e-10 & 7.3e-11 & 2.8e+04 & 5.2e-10 & 7.5e-12 \\ \hline
 16 & 5.9e+00 & 7.4e+00 & 2.0e-10 & 4.5e-10 & 7.7e-11 & 1.8e+04 & 4.1e-10 & 5.8e-12 \\ \hline
 32 & 8.0e+00 & 1.1e+01 & 3.3e-10 & 4.0e-10 & 7.9e-11 & 1.8e+04 & 3.7e-10 & 6.4e-12 \\ \hline
 64 & 9.7e+00 & 1.6e+01 & 4.1e-10 & 3.4e-10 & 5.3e-11 & 1.9e+04 & 3.1e-10 & 5.9e-12 \\ \hline
 128 & 1.6e+01 & 2.3e+01 & 4.7e-10 & 3.2e-10 & 3.9e-11 & 1.7e+04 & 2.9e-10 & 4.3e-12 \\ \hline
 256 & 2.2e+01 & 3.2e+01 & 5.7e-10 & 4.3e-10 & 3.3e-11 & 1.7e+04 & 3.9e-10 & 3.4e-12
  \\ \bottomrule
\end{tabular}
}
\caption{\small Bound vs. $m$, the number of matrices in $\cA$ for  $\tau_6=(1,2,3)$}
\label{tab2}
\end{table}
We summarize our observations from Tables~\ref{tab1} and \ref{tab2} as follows.
\begin{enumerate}
\item For all $m$, the {\em ratio}s are far less than $1$.
      In the other word, \eqref{eq:pert-tiny-cond} is satisfied for all, and hence the bound \eqref{ineq:main} holds.

\item For all $m$, $\varepsilon_{\ub}$ provides a very good upper bound on the {\em error}.

\item As $m$ increases, i.e., as we expand the matrix set $\cA$,
      the modulus of uniqueness and modulus of non-divisibility increase as well,
      and the condition number $\cond(\cA)$ decreases at first, then remains almost the same.
\end{enumerate}

\smallskip
\noindent{\bf Test 2: matrix sizes.}\quad
In this test, we fix $\xi=10^{-12}$, $m=16$, and use two partitions $\tau_n=p\times (3,3,3)$
or $\tau_n=p\times (1,2,3)$, where $p=1,2,\dots,7$.
Then the matrix size $n=9p$ or $6p$ will increase as $p$ increases.
We display the numerical results in Tables~\ref{tab3} and \ref{tab4}.
We can see from  Tables~\ref{tab3} and \ref{tab4} that
$\varepsilon_{\ub}$ provides a very good upper bound on the {\em error} for different sizes of matrices.

\begin{table}[ht]
\belowrulesep=1pt
\aboverulesep=1pt
  \centering
{\small
\begin{tabular}{c||c|c|c|c|c|c|c|c}\toprule
 $n$ & $\omega_{\uniq}$ & $\omega_{\robu}$  & $\delta$ & {\em ratio\/} & $\varepsilon_{\berr}$ &  $\cond(\cA)$ & $\varepsilon_{\ub}$ & {\em error} \\ \hline\hline
 9 & 6.8e+00 & 6.8e+00 & 2.1e-10 & 7.7e-10 & 4.5e-11 & 2.4e+02 & 7.1e-10 & 3.9e-12 \\ \hline
 18 & 1.1e+01 & 1.1e+01 & 2.5e-09 & 2.1e-09 & 1.3e-09 & 6.3e+03 & 2.0e-09 & 5.6e-11 \\ \hline
 27 & 1.2e+01 & 1.2e+01 & 1.1e-08 & 5.1e-09 & 4.3e-09 & 1.7e+04 & 4.7e-09 & 1.2e-10 \\ \hline
 36 & 1.4e+01 & 1.4e+01 & 6.7e-09 & 2.3e-09 & 1.2e-09 & 5.6e+03 & 2.1e-09 & 3.2e-11 \\ \hline
 45 & 1.6e+01 & 1.6e+01 & 3.1e-09 & 2.0e-09 & 1.2e-09 & 4.4e+03 & 1.8e-09 & 1.8e-11 \\ \hline
 54 & 1.8e+01 & 1.8e+01 & 1.7e-08 & 4.7e-09 & 6.1e-09 & 2.6e+04 & 4.4e-09 & 5.7e-11 \\ \hline
 63 & 1.9e+01 & 1.9e+01 & 2.1e-07 & 5.4e-08 & 7.2e-08 & 9.4e+03 & 5.0e-08 & 7.7e-10
  \\ \bottomrule
\end{tabular}
}
\caption{\small Bound vs. matrix size $n=9p$ for  $\tau_n=p\times (3,3,3)$}
\label{tab3}
\end{table}

\begin{table}[ht]
\belowrulesep=1pt
\aboverulesep=1pt
  \centering
{\small
\begin{tabular}{c||c|c|c|c|c|c|c|c}\toprule
 $n$ & $\omega_{\uniq}$ & $\omega_{\robu}$  & $\delta$ & {\em ratio\/} & $\varepsilon_{\berr}$ &  $\cond(\cA)$ & $\varepsilon_{\ub}$ & {\em error} \\ \hline\hline
 6 & 4.2e+00 & 5.7e+00 & 1.8e-10 & 3.7e-10 & 2.6e-11 & 1.0e+02 & 3.4e-10 & 4.6e-12 \\ \hline
 12 & 6.8e+00 & 6.7e+00 & 3.5e-10 & 7.9e-10 & 7.6e-11 & 4.8e+02 & 7.3e-10 & 6.0e-12 \\ \hline
 18 & 8.8e+00 & 9.4e+00 & 5.7e-10 & 1.6e-09 & 3.5e-10 & 5.5e+03 & 1.4e-09 & 1.2e-11 \\ \hline
 24 & 9.0e+00 & 8.5e+00 & 4.7e-09 & 3.1e-09 & 1.5e-09 & 4.4e+03 & 2.8e-09 & 5.0e-11 \\ \hline
 30 & 9.5e+00 & 9.0e+00 & 9.2e-09 & 4.8e-09 & 3.6e-09 & 7.2e+03 & 4.4e-09 & 5.5e-11 \\ \hline
 36 & 1.2e+01 & 1.0e+01 & 3.8e-09 & 4.4e-09 & 2.3e-09 & 1.9e+03 & 4.1e-09 & 4.4e-11 \\ \hline
 42 & 1.3e+01 & 1.2e+01 & 6.9e-09 & 4.7e-09 & 6.5e-09 & 1.2e+05 & 4.4e-09 & 4.5e-11
  \\ \bottomrule
\end{tabular}
}
\caption{\small Bound vs. matrix size $n=6p$ for  $\tau_n=p\times (1,2,3)$}
\label{tab4}
\end{table}

\smallskip
\noindent{\bf Test 3: number of diagonal blocks.}\quad
In this test, we fix $\xi=10^{-12}$, $m=16$, and
generate the partition $\tau_n$ randomly using {\sc matlab} command {\tt randi(5,t,1)}.
In the other word, the block diagonal matrices $D_j$ have $t$ diagonal blocks and
the order of the $i$th block is $\tau_n(i)$, randomly drawn from $\{1,2,\ldots,5\}$ with the uniform  distribution.
For $t=3,4,\dots,9$,
we display the numerical results in Table~\ref{tab5}.
We can see from  Table~\ref{tab5}  that
$\varepsilon_{\ub}$ provides a very good upper bound on the {\em error} for the different numbers of diagonal blocks.

\begin{table}[ht]
\belowrulesep=1pt
\aboverulesep=1pt
  \centering
{\small
\begin{tabular}{c||c|c|c|c|c|c|c|c}\toprule
 $t$ & $\omega_{\uniq}$ & $\omega_{\robu}$  & $\delta$ & {\em ratio\/} & $\varepsilon_{\berr}$ &  $\cond(\cA)$ & $\varepsilon_{\ub}$ & {\em error} \\ \hline\hline
 3 & 5.7e+00 & 7.6e+00 & 6.7e-10 & 5.9e-10 & 1.9e-10 & 1.8e+04 & 5.4e-10 & 1.1e-11 \\ \hline
 4 & 3.5e+00 & 7.1e+00 & 5.7e-10 & 4.1e-09 & 6.2e-10 & 4.2e+03 & 3.7e-09 & 5.2e-11 \\ \hline
 5 & 3.8e+00 & 5.8e+00 & 8.3e-10 & 3.8e-09 & 8.1e-10 & 4.4e+03 & 3.3e-09 & 1.8e-11 \\ \hline
 6 & 4.0e+00 & 6.0e+00 & 8.0e-10 & 3.5e-09 & 6.7e-10 & 2.2e+04 & 3.0e-09 & 1.2e-11 \\ \hline
 7 & 5.8e+00 & 6.5e+00 & 1.9e-09 & 7.1e-09 & 2.7e-09 & 1.2e+04 & 6.1e-09 & 3.7e-11 \\ \hline
 8 & 4.4e+00 & 8.1e+00 & 2.4e-09 & 1.5e-08 & 3.0e-09 & 3.5e+04 & 1.3e-08 & 3.6e-11 \\ \hline
 9 & 3.9e+00 & 8.4e+00 & 1.1e-09 & 9.5e-09 & 8.7e-10 & 1.3e+04 & 8.1e-09 & 1.3e-11
  \\ \bottomrule
\end{tabular}
}
\caption{\small Bound vs. number of diagonal blocks}
\label{tab5}
\end{table}

\smallskip
\noindent{\bf Test 4: noise level.}\quad
In this test, we fix the number of matrices $m=16$.
For different partitions $\tau_n=(3,3,3)$ and $\tau_n=(1,2,3)$,
in Figure~\ref{fig:theta},
we plot $\varepsilon_{\berr}$ ({\em backward error\/}), {\em error} and $\varepsilon_{\ub}$ ({\em bound\/}) versus different noise levels.
We can see from Figure~\ref{fig:theta} that as $\xi$ increases,
$\varepsilon_{\berr}$, {\em error} and $\varepsilon_{\ub}$  all increase almost  linearly.
For all noise levels,  $\varepsilon_{\ub}$  indeed provides a good upper bound on the {\em error}.

\begin{figure}[ht]
\centering
\includegraphics[width=0.49\textwidth]{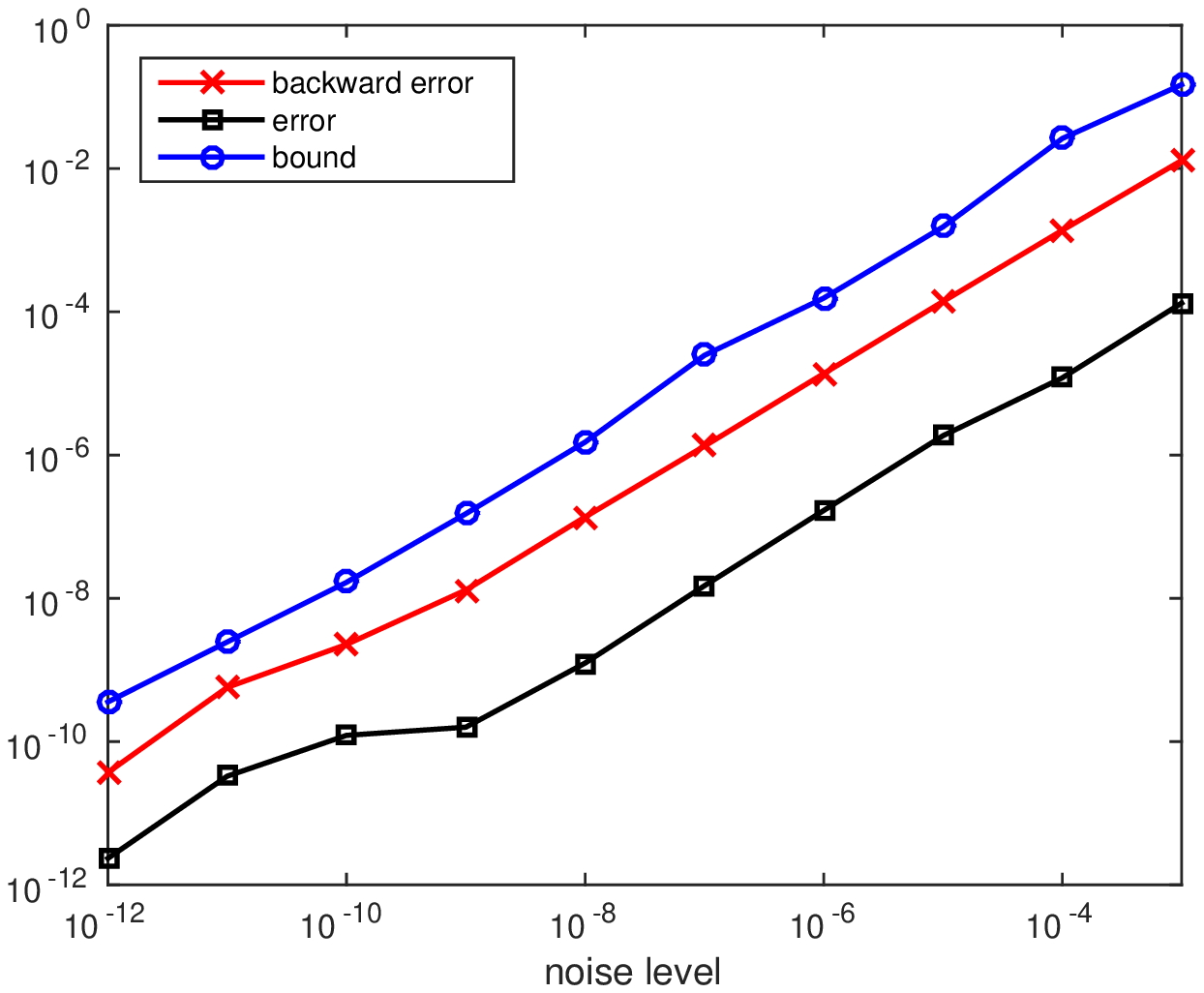}
\includegraphics[width=0.49\textwidth]{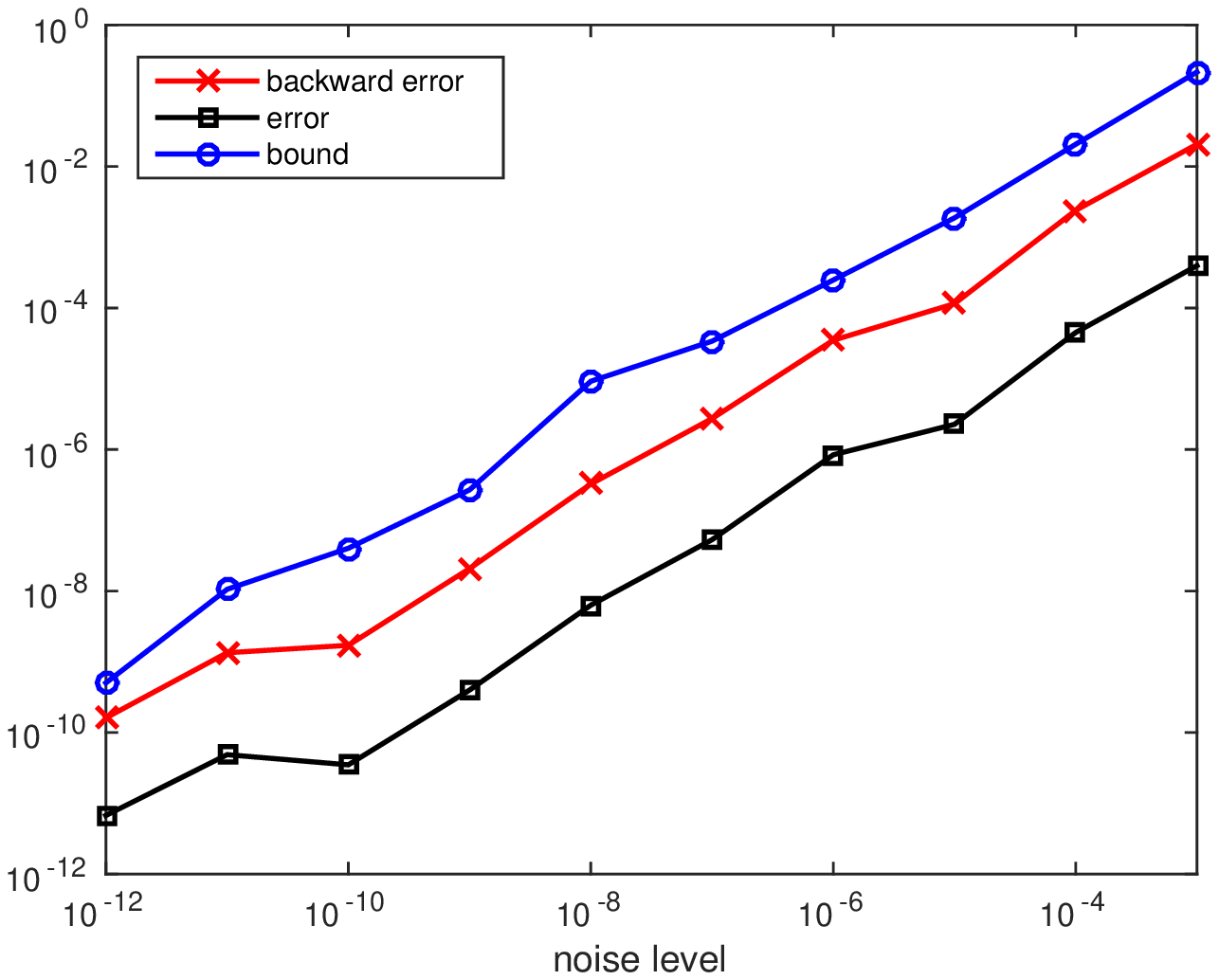}
$\tau_n=(3,3,3)$\hspace{1.5in} $\tau_n=(1,2,3)$
\caption{\small Backward error $\varepsilon_{\berr}$, {\em error}, and  bound $\varepsilon_{\ub}$ vs. noise level}
\label{fig:theta}
\end{figure}

\smallskip
\noindent{\bf Test 5: condition number.}\quad
In this test, we fix $m=16$, $\xi=10^{-12}$.
For two different partitions $\tau_n=(3,3,3)$ and $\tau_n=(1,2,3)$,
we ran the tests 100 times for each partition.
In Figure~\ref{fig:con},
we plot  the quotient $\varepsilon_{\ub}$/{\em error} versus the condition number $\cond(\cA)$.
The smaller the quotient is, the sharper $\varepsilon_{\ub}$ estimates the {\em error}.
We can see from Figure~\ref{fig:con} that $\varepsilon_{\ub}$ provides a good
upper bound on the {\em error}, even as the condition number becomes large.

\begin{figure}[ht]
\centering
\includegraphics[width=0.49\textwidth]{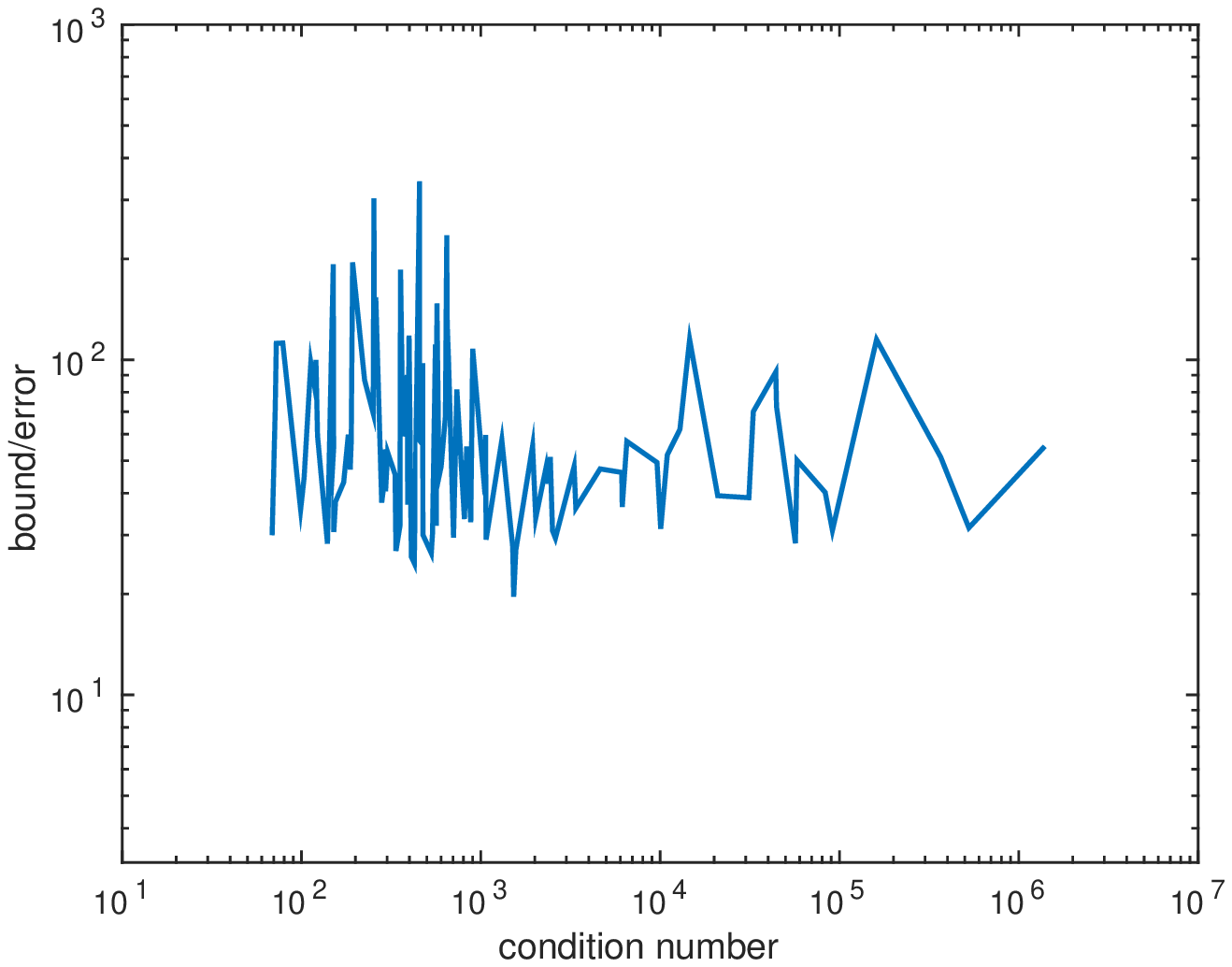}
\includegraphics[width=0.49\textwidth]{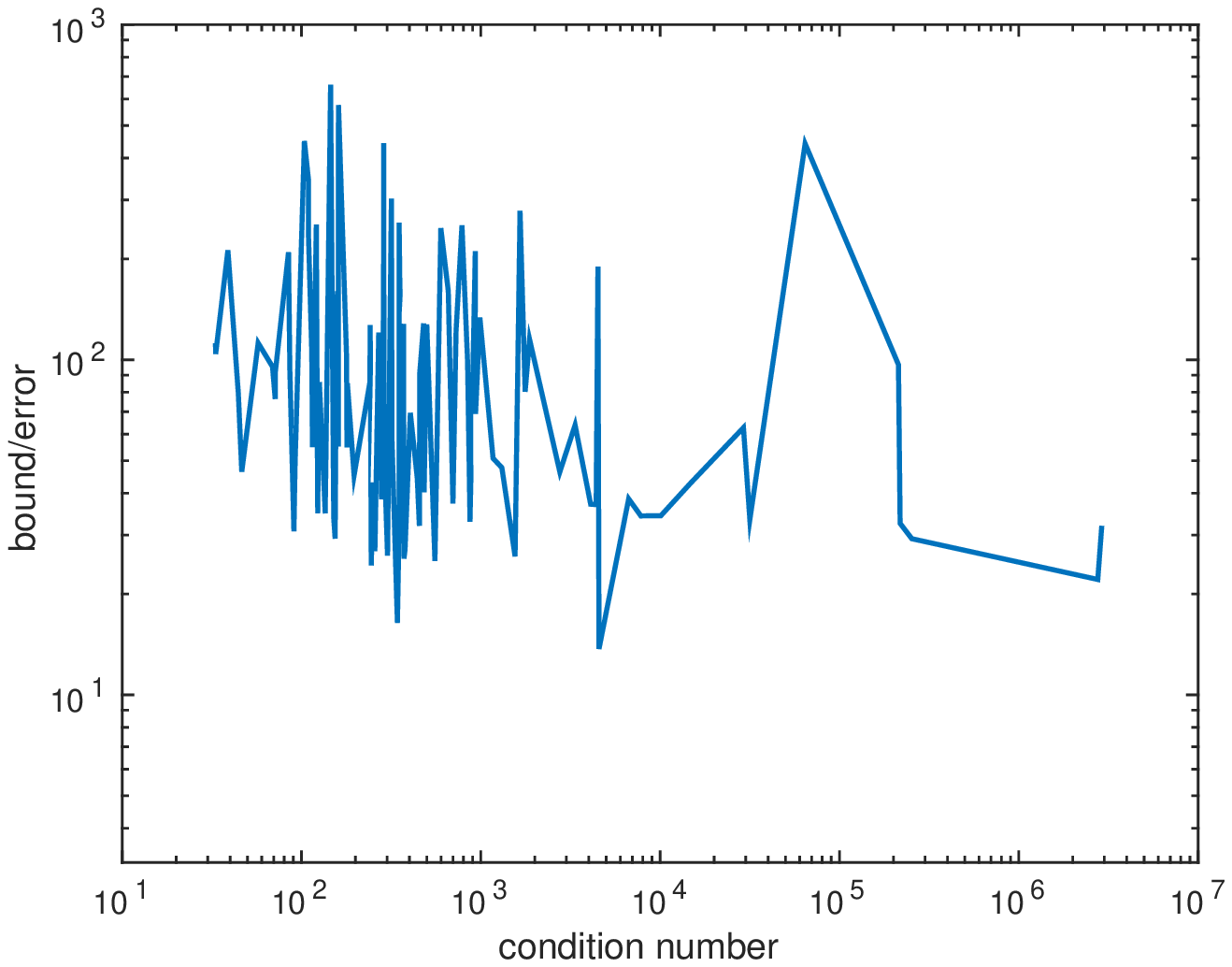}\\
$\tau_n=(3,3,3)$\hspace{1.5in} $\tau_n=(1,2,3)$
\caption{\small Bound $\varepsilon_{\ub}$/{\em error} vs. condition number $\cond(\cA)$}
\label{fig:con}
\end{figure}

%\end{example}

\section{Concluding Remarks}\label{sec:conclusion}
In this paper, we developed a perturbation theory for \jbdp.
An upper bound is obtained for the relative distance \eqref{eq:measure}
%\begin{equation}\tag{\ref{eq:measure}}
%\min_{D\in\mathbb{D}_{\tau_n},  \Pi\in\mathbb{P}_{\tau_n}}\frac{\|W - \ww D\Pi\|}{\|\ww\|}
%\end{equation}
between a block diagonalizer $W$ for the original \jbdp\ of $\cA$
that is block diagonalizable and an approximate diagonalizer $\wtd W$ for its perturbed \jbdp\ of $\wtd\cA$.
The backward error and condition number are also derived and discussed for \jbdp.
Numerical tests validate the theoretical results.

The \jbdp\ of interest in this paper is for block diagonalization via
congruence transformations which are known to preserve symmetry. Yet our development
so far does not assume that all $A_i$ are symmetric. What will happen to all the results
if they are symmetric? It turns out that not much simplification in results and arguments can be
gained but
all the results remain valid
after minor changes to the definitions of $G_{jk}$ in \eqref{hatmjk}: remove
the second, fourth, $\ldots$, block rows as now all $A_i^{(jj)}$ are symmetric.

We have been limiting all the matrices to real ones, but this is not a limitation.
In fact, if all matrices are complex, the change that needs to be made is simply
to replace all transposes $\scriptstyle\T$ by complex conjugate transposes
$\scriptstyle\HH$, but for simplicity we still would like to keep
all $\gamma_i$, the diagonal entries of $\Gamma$ real, so that we don't have to
change the definition of the gap $g$ in \eqref{eq:res-err}.

Conceivably, we might use similarity transformation for block diagonalization, i.e., instead of
\eqref{eq:nojbd}, we may seek a nonsingular matrix $W\in\mathbb{R}^{n\times n}$ such that all $W^{-1} A_i W$
are $\tau_n$-block diagonal. A similar development that are very much parallel to those in
\cite{cai2017algebraic} and in this paper can be worked out. A major change will be to redefine
the subspace $\na$ in \eqref{na} as
$$
\na:=\big\{Z\in\mathbb{R}^{n\times n}\; :\; A_iZ-ZA_i=0\,\,\mbox{for $1\le i\le m$}\big\}.
$$
We omit the detail.

%{\small
%\bibliographystyle{abbrv}
%%    Insert the bibliography data here.
%% used by Li
%\bibliography{myref}
%% used by Cai
%%\bibliography{myref}

\begin{thebibliography}{10}

\bibitem{afsari2008sensitivity}
B.~Afsari.
\newblock Sensitivity analysis for the problem of matrix joint diagonalization.
\newblock {\em SIAM J. Matrix Anal. Appl.}, 30(3):1148--1171, 2008.

\bibitem{bai2009exploiting}
Y.~Bai, E.~de~Klerk, D.~Pasechnik, and R.~Sotirov.
\newblock Exploiting group symmetry in truss topology optimization.
\newblock {\em Optim. Engrg.}, 10(3):331--349, 2009.

\bibitem{cai2017algebraic}
Y.~Cai and C.~Liu.
\newblock An algebraic approach to nonorthogonal general joint block
  diagonalization.
\newblock {\em SIAM J. Matrix Anal. Appl.}, 38(1):50--71, 2017.

\bibitem{cardoso1998multidimensional}
J.-F. Cardoso.
\newblock Multidimensional independent component analysis.
\newblock In {\em Acoustics, Speech and Signal Processing, 1998. Proceedings of
  the 1998 IEEE International Conference on}, volume~4, pages 1941--1944. IEEE,
  Washinton, DC, 1998.

\bibitem{chabriel2014joint}
G.~Chabriel, M.~Kleinsteuber, E.~Moreau, H.~Shen, P.~Tichavsky, and A.~Yeredor.
\newblock Joint matrices decompositions and blind source separation: A survey
  of methods, identification, and applications.
\newblock {\em IEEE Signal Process. Mag.}, 31(3):34--43, 2014.

\bibitem{de2007reduction}
E.~De~Klerk, D.~V. Pasechnik, and A.~Schrijver.
\newblock Reduction of symmetric semidefinite programs using the regular
  $\ast$-representation.
\newblock {\em Math. Program.}, 109(2-3):613--624, 2007.

\bibitem{de2010exploiting}
E.~De~Klerk and R.~Sotirov.
\newblock Exploiting group symmetry in semidefinite programming relaxations of
  the quadratic assignment problem.
\newblock {\em Math. Program.}, 122(2):225--246, 2010.

\bibitem{de2008decompositions}
L.~De~Lathauwer.
\newblock Decompositions of a higher-order tensor in block terms-part {I}:
  Lemmas for partitioned matrices.
\newblock {\em SIAM J. Matrix Anal. Appl.}, 30(3):1022--1032, 2008.

\bibitem{de2008decompositions2}
L.~De~Lathauwer.
\newblock Decompositions of a higher-order tensor in block terms-part {II}:
  Definitions and uniqueness.
\newblock {\em SIAM J. Matrix Anal. Appl.}, 30(3):1033--1066, 2008.

\bibitem{de2009survey}
L.~De~Lathauwer.
\newblock A survey of tensor methods.
\newblock In {\em 2009 IEEE International Symposium on Circuits and Systems},
  pages 2773--2776. IEEE, 2009.

\bibitem{de2000fetal}
L.~De~Lathauwer, B.~De~Moor, and J.~Vandewalle.
\newblock Fetal electrocardiogram extraction by blind source subspace
  separation.
\newblock {\em IEEE Trans. Biomedical Engrg.}, 47(5):567--572, 2000.

\bibitem{de2008decompositions3}
L.~De~Lathauwer and D.~Nion.
\newblock Decompositions of a higher-order tensor in block terms-part {III}:
  Alternating least squares algorithms.
\newblock {\em SIAM J. Matrix Anal. Appl.}, 30(3):1067--1083, 2008.

\bibitem{demmel1997applied}
J.~W. Demmel.
\newblock {\em Applied Numerical Linear Algebra}.
\newblock SIAM, Philadelphia, PA, 1997.

\bibitem{domanov2013uniqueness}
I.~Domanov and L.~De~Lathauwer.
\newblock On the uniqueness of the canonical polyadic decomposition of
  third-order tensors--part {I}: Basic results and uniqueness of one factor
  matrix.
\newblock {\em SIAM J. Matrix Anal. Appl.}, 34(3):855--875, 2013.

\bibitem{domanov2013uniqueness2}
I.~Domanov and L.~De~Lathauwer.
\newblock On the uniqueness of the canonical polyadic decomposition of
  third-order tensors--part {II}: Uniqueness of the overall decomposition.
\newblock {\em SIAM J. Matrix Anal. Appl.}, 34(3):876--903, 2013.

\bibitem{gatermann2004symmetry}
K.~Gatermann and P.~A. Parrilo.
\newblock Symmetry groups, semidefinite programs, and sums of squares.
\newblock {\em J. Pure Appl. Algebra}, 192(1):95--128, 2004.

\bibitem{kahan1975spectra}
W.~Kahan.
\newblock Spectra of nearly hermitian matrices.
\newblock {\em Proc. Amer. Math. Soc.}, 48(1):11--17, 1975.

\bibitem{kruskal1977three}
J.~B. Kruskal.
\newblock Three-way arrays: rank and uniqueness of trilinear decompositions,
  with application to arithmetic complexity and statistics.
\newblock {\em Linear Algebra Appl.}, 18(2):95--138, 1977.

\bibitem{li:2014HLA}
R.-C. Li.
\newblock Matrix perturbation theory.
\newblock In L.~Hogben, R.~Brualdi, and G.~W. Stewart, editors, {\em Handbook
  of Linear Algebra}, chapter~21. CRC Press, Boca Raton, FL, 2nd edition, 2014.

\bibitem{nion2011tensor}
D.~Nion.
\newblock A tensor framework for nonunitary joint block diagonalization.
\newblock {\em IEEE Trans. Signal Process.}, 59(10):4585--4594, 2011.

\bibitem{poczos2005independent}
B.~P{\'o}czos and A.~L{\H{o}}rincz.
\newblock Independent subspace analysis using k-nearest neighborhood distances.
\newblock In {\em Artificial Neural Networks: Formal Models and Their
  Applications-ICANN 2005}, pages 163--168. Springer, 2005.

\bibitem{russo2011argument}
F.~G. Russo.
\newblock On an argument of j.-f. cardoso dealing with perturbations of joint
  diagonalizers.
\newblock 2011.
\newblock Available at {\tt arXiv:1103.3670}.

\bibitem{shi2015some}
D.~C. Shi, Y.~F. Cai, and S.~F. Xu.
\newblock Some perturbation results for a normalized non-orthogonal joint
  diagonalization problem.
\newblock {\em Linear Algebra Appl.}, 484:457--476, 2015.

\bibitem{sorensen2015coupled}
M.~S{\o}rensen and L.~De~Lathauwer.
\newblock Coupled canonical polyadic decompositions and (coupled)
  decompositions in multilinear rank-({L}r,n,{L}r,n,1) terms--part {I}:
  Uniqueness.
\newblock {\em SIAM J. Matrix Anal. Appl.}, 36(2):496--522, 2015.

\bibitem{sorensen2015new}
M.~S{\o}rensen and L.~De~Lathauwer.
\newblock New uniqueness conditions for the canonical polyadic decomposition of
  third-order tensors.
\newblock {\em SIAM J. Matrix Anal. Appl.}, 36(4):1381--1403, 2015.

\bibitem{stegeman2011uniqueness}
A.~Stegeman.
\newblock On uniqueness of the canonical tensor decomposition with some form of
  symmetry.
\newblock {\em SIAM J. Matrix Anal. Appl.}, 32(2):561--583, 2011.

\bibitem{stewart1990matrix}
G.~W. Stewart and J.-G. Sun.
\newblock {\em Matrix Perturbation Theory}.
\newblock Academic Press, Boston, 1990.

\bibitem{sun1996variation}
J.~G. Sun.
\newblock On the variation of the spectrum of a normal matrix.
\newblock {\em Linear Algebra Appl.}, 246:215 -- 223, 1996.

\bibitem{theis2005blind}
F.~J. Theis.
\newblock Blind signal separation into groups of dependent signals using joint
  block diagonalization.
\newblock In {\em Circuits and Systems, 2005. ISCAS 2005. IEEE International
  Symposium on}, pages 5878--5881. IEEE, 2005.

\bibitem{theis2006towards}
F.~J. Theis.
\newblock Towards a general independent subspace analysis.
\newblock In {\em Advances in Neural Information Processing Systems}, pages
  1361--1368, MIT Press, Cambridge, MA, 2006.

\bibitem{tichavsky2014non}
P.~Tichavsky, A.~H. Phan, and A.~Cichocki.
\newblock Non-orthogonal tensor diagonalization.
\newblock 2014.
\newblock Available at {\tt arXiv:1402.1673v3}.

\bibitem{van1996matrix}
C.~F. Van~Loan and G.~H. Golub.
\newblock {\em Matrix computations}.
\newblock Johns Hopkins University Press, Baltimore, MD, 4th edition, 2012.

\bibitem{vannieuwenhoven2016condition}
N.~Vannieuwenhoven.
\newblock A condition number for the tensor rank decomposition.
\newblock 2016.
\newblock Available at {\tt arXiv:1604.00052}.

\bibitem{tensorlab}
N.~Vervliet, O.~Debals, L.~Sorber, M.~Van~Barel, and L.~De~Lathauwer.
\newblock Tensorlab 3.0, March 2016.
\newblock Available at {\tt www.tensorlab.net}.

\end{thebibliography}
%}

%\blindtext
\end{document}